\documentclass[leqno,11pt,a4paper]{amsart}
\usepackage[margin=2.7cm]{geometry}
\renewcommand{\baselinestretch}{1.2}

\usepackage{pdfsync}
\usepackage[usenames,dvipsnames]{color}
\usepackage{hyperref}
\usepackage{mleftright}
\usepackage{enumitem}
\usepackage{amssymb}
\usepackage{mathrsfs}
\usepackage{tikz}
\usepackage{caption}
\usetikzlibrary{decorations.markings}
\newtheorem{thm}{Theorem}[section]
\newtheorem{lemma}[thm]{Lemma}
\newtheorem{cor}[thm]{Corollary}
\newtheorem{prop}[thm]{Proposition}
\theoremstyle{definition}
\newtheorem{example}[thm]{Example}
\newtheorem{remark}[thm]{Remark}

\newtheorem{definition}[thm]{Definition}

\newtheorem{notation}[thm]{Notation}

\numberwithin{equation}{section}
\long\def\blankfootnotetext#1{\begingroup\def\thefootnote{\fnsymbol{footnote}}\footnotetext{#1}\endgroup}
\newcommand{\orig}{\mathbf{0}}
\newcommand{\Z}{\mathbb{Z}}

\newcommand{\Q}{\mathbb{Q}}

\newcommand{\Proj}{\mathbb{P}}

\newcommand{\Hom}[1]{\mathrm{Hom}\mleft({#1}\mright)}

\renewcommand{\gcd}[1]{\mathrm{gcd}\mleft\{{#1}\mright\}}

\newcommand{\cone}[1]{\mathrm{cone}\mleft({#1}\mright)}
\newcommand{\conv}[1]{\mathrm{conv}\mleft({#1}\mright)}

\newcommand{\dual}[1]{{#1}^\vee}
\newcommand{\mut}{\mathrm{mut}}
\newcommand{\NQ}{N_\Q}
\newcommand{\MQ}{M_\Q}

\newcommand{\quiv}{\mathrm{\textbf{quiv}}}
\newcommand{\bquiv}{\mathrm{\textbf{bquiv}}}
\newcommand{\fanos}{\mathfrak{F}}

\newcommand{\verts}[1]{\mathrm{verts}\mleft(#1\mright)}
\newcommand{\spanf}[1]{\Sigma_{#1}}
\newcommand{\sref}{\Sigma}
\newcommand{\hmin}{h_\mathrm{min}}
\newcommand{\hmax}{h_\mathrm{max}}
\newcommand{\arr}[1]{\mathrm{a}(#1)}
\newcommand{\vol}{\mathrm{vol}}
\newcommand{\into}[1]{\mathrm{in}\mleft(#1\mright)}
\newcommand{\outof}[1]{\mathrm{out}\mleft(#1\mright)}
\newcommand{\basket}{\mathcal{B}}
\newcommand{\lab}[1]{\scriptstyle{\mathbf{(#1)}}}
\newcommand{\radist}[1]{r(#1)}
\newcommand{\seq}[1]{\mathrm{seq}(#1)}
\newcommand{\ind}[1]{\mathrm{g}(#1)}

\begin{document}
\tikzset{middlearrow/.style={
        decoration={markings,
            mark= at position 0.6 with {\arrow{#1}} ,
        },
        postaction={decorate}
    }
}
\author[Mohammad E.~Akhtar]{Mohammad E.~Akhtar}
\blankfootnotetext{2010 \emph{Mathematics Subject Classification}: 14J45 (Primary); 52B20, 14J33 (Secondary).}
\title{Polygonal Quivers}
\maketitle
\begin{abstract}
We show that Fano lattice polygons define a class of balanced quivers with interesting properties. The combinatorics of these quivers is related to singularities of the underlying toric Fano surface. This allows us to show that every Fano polygon defines a point on a certain family of algebraic hypersurfaces. Our quivers admit a generalized mutation which preserves balancing and coincides with combinatorial mutation of Fano polygons whenever both operations are defined. We characterize balanced quivers arising from Fano polygons and discuss generalizations to higher dimensions.
\end{abstract}
\section{Opening Remarks}
\subsection{Overview}\label{subsec:overview}
Fano polytopes and their combinatorial mutations occupy a central position in the recent programme~\cite{CCGGK,CCGK} to classify Fano varieties using mirror symmetry. In the case of surfaces, there is expected to be a one-to-one correspondence~\cite[Conjecture~A]{AMany} between mutation classes of Fano polygons and $\Q$-Gorenstein deformation classes of Fano orbifolds.

This discussion is about quivers arising from Fano polygons. In Section~\ref{sec:main_construction} we define the class of \emph{polygonal quivers}. Our definition brings together, and extends, earlier proposals appearing in theoretical physics~\cite{FHHU,HV05} and classification theory~\cite{KNP15}. The polygonal quivers of $\Proj^2$ and $\Proj(1,1,6)$ are shown in Figure~\ref{fig:P2_P116_quivers}.
\begin{figure}[h!]
    \centering
    \begin{tikzpicture}[thick,scale=0.8, every node/.style={transform shape}]
      \coordinate  (P2v1) at (-2, -1);
      \coordinate [label={below:\small $(1,1)$}]  (P2v1_lab) at (-2, -1.15);
      \coordinate  (P2v2) at (-4, -1);
      \coordinate [label={below:\small $(1,1)$}]  (P2v2_lab) at (-4, -1.15);
      \coordinate  (P2v3) at (-3, 0.5);
      \coordinate [label={above:\small $(1,1)$}]  (P2v3_lab) at (-3, 0.65);

      \coordinate  (P116v1) at (4, -1);
      \coordinate [label={below:\small $(2,3)$}]  (P116v1_lab) at (4, -1.15);
      \coordinate  (P116v2) at (2, -1);
      \coordinate [label={below:\small $(1,1)$}]  (P116v2_lab) at (2, -1.15);
      \coordinate  (P116v3) at (3, 0.5);
      \coordinate [label={above:\small $(1,1)$}]  (P116v3_lab) at (3, 0.65);

      \draw [fill=black] (P2v1) circle (0.05);
      \draw (P2v1) circle (0.1);
      \draw [fill=black] (P2v2) circle (0.05);
      \draw (P2v2) circle (0.1);
      \draw [fill=black] (P2v3) circle (0.05);
      \draw (P2v3) circle (0.1);

      \draw [fill=black] (P116v1) circle (0.05);
      \draw (P116v1) circle (0.1);
      \draw [fill=black] (P116v2) circle (0.05);
      \draw (P116v2) circle (0.1);
      \draw [fill=black] (P116v3) circle (0.05);
      \draw (P116v3) circle (0.1);

      \draw [middlearrow={latex}] (P2v2) -- node[below] {$3$} (P2v1);
      \draw [middlearrow={latex}] (P2v1) -- node[right] {\hspace{1mm}$3$} (P2v3);
      \draw [middlearrow={latex}] (P2v3) -- node[left] {\hspace{-6mm}$3$} (P2v2);

      \draw [middlearrow={latex}] (P116v2) -- node[below] {$4$} (P116v1);
      \draw [middlearrow={latex}] (P116v1) -- node[right] {\hspace{1mm}$4$} (P116v3);
      \draw [middlearrow={latex}] (P116v3) -- node[left] {\hspace{-6mm}$8$} (P116v2);

   \end{tikzpicture}
   \caption{The Polygonal Quivers for $\Proj^{2}$ (left) and $\Proj(1,1,6)$ (right).}\label{fig:P2_P116_quivers}
\end{figure}
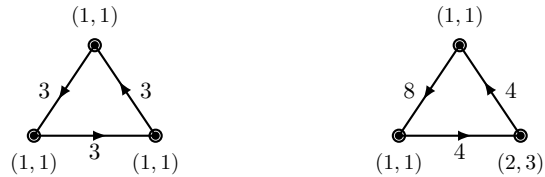

Notice that our quivers are decorated: each vertex is labelled by a pair $(w,\ell) \in \Z^2$. They also contain no self-loops or $2$-cycles (Example~\ref{eg:no_self_loops_or_2-cycles}). In Section~\ref{sec:balancing_condition}, we establish the \emph{balancing condition} for polygonal quivers (Proposition~\ref{prop:balancing_condition}). A special case of this is already known for reflexive polygons: the number of arrows into a given vertex is equal to the number of arrows out of that vertex. We extend this to all Fano polygons, by considering arrows weighted by vertex labels. This has a natural interpretation in terms of diameters of Fano polygons.

Section~\ref{sec:mutations} is about mutations: we extend classical quiver mutation~\cite{Sei94,FZ1} to balanced quivers (decorated quivers which satisfy the balancing condition). The class of balanced quivers is closed under this extended mutation (Proposition~\ref{prop:balanced_quivers_are_closed_under_mutation}), but a polygonal quiver may mutate to something non-polygonal (Example~\ref{eg:polygonal_quivers_are_not_closed_under_mutation}). We characterize precisely when this happens (Proposition~\ref{prop:quiv_mutation_commutes_with_comb_mutation_at_Tcones}) by relating our extended quiver mutations to combinatorial mutations of Fano polygons~\cite{ACGK12}. From this perspective, the failure to remain polygonal is measured by the number of vertices corresponding to residual singularities, introduced in~\cite{AK14}. We discuss mutation invariants coming from the arrows and vertex labels of our quivers (Propositions~\ref{prop:arrow_invariant} and~\ref{prop:weight_invariant}) and observe a group structure on mutations (Proposition~\ref{prop:mutation_group}).

Section~\ref{sec:quiver_degree_formula} adopts a more geometric viewpoint: every Fano polygon defines a toric Fano surface. We compute the anticanonical degree of this surface in terms of polygonal quivers. Combining this with the Noether-type formula given in~\cite[Proposition~3.3]{AK14} we obtain the \emph{quiver degree formula}~\eqref{eqn:quiver_degree_formula}, relating the singularities of a toric Fano surface to the combinatorics of its polygonal quiver. This allows us to establish non-existence results (Example~\ref{eg:non_existence_of_fanos_with_prescribed_singularities}) and show that every Fano polygon defines a point on a certain family of algebraic hypersurfaces (Proposition~\ref{prop:markov_varieties}). This generalizes earlier results~\cite{HP10,AK13} relating Fano triangles to solutions of Markov-type Diophantine equations.

In Section~\ref{sec:reconstruction_and_characterization} we characterize balanced quivers arising from Fano polygons (Theorem~\ref{thm:reconstruction_general}). The case of triangles is simpler, and is treated separately (Proposition~\ref{prop:reconstruction_triangles}). We use the notion of \emph{expected volume}~\eqref{eqn:expected_volume_triangles} arising in the triangle setting to demonstrate families of non-polygonal quivers (Example~\ref{eg:family_of_non_polygonal_quivers}). The paper concludes with a view towards higher dimensions.

\begin{remark}
Unless otherwise stated, we will only consider quivers with finitely many vertices and finitely many arrows between any pair of vertices. This condition is automatically satisfied for polygonal and block-polygonal quivers (Definitions~\ref{def:quiv_definition} and~\ref{def:bquiv_definition}). It allows us to avoid convergence issues, for instance in the definition of balanced quivers appearing in Section~\ref{sec:mutations}. A number of definitions and results, particularly those in Section~\ref{sec:mutations}, would remain valid (with essentially the same proofs) if finiteness of vertices was replaced with the condition that the set 
\end{remark}

\subsection{Background and Notation}\label{subsec:background}
Let $N$ be a lattice of rank $2$. A $2$-dimensional lattice polytope $P \subset \NQ:=N \otimes_{\Z} \Q$ which contains the origin in its strict interior is called a \emph{Fano polygon}~\cite{KN12} if every vertex $v$ of $P$ is primitive: $\conv{\orig,v} \cap N = \{\orig,v\}$. The \emph{spanning fan} of a Fano polygon $P$ is the following complete fan in $\NQ$:
\[
    \spanf{P}:=\{\mathrm{cone}(\tau) \subset \NQ \mid \tau \text{ is a proper face of }P\}.
\]
If $\sigma \subset \NQ$ is a $2$-dimensional strictly convex, rational polyhedral cone with primitive generators $u,v$ then the \emph{inner normal vector} of $\sigma$ is the unique primitive vector $m_{\sigma} \in M:= \Hom{N,\Z}$ which defines the hyperplane containing both $u,v$ and satisfies $\langle m_{\sigma},u\rangle = \langle m_{\sigma},v\rangle < 0$. The integer $\ell_{\sigma}:= - \langle m_{\sigma},u\rangle$ is the \emph{local index} of $\sigma$ and $w_{\sigma}:= \mleft|\conv{u,v} \cap N\mright| - 1$ is called the \emph{width} of $\sigma$. Consider the following division, which depends only on $\sigma$:
\begin{equation}\label{eqn:width_divided_by_local_index}
    w_{\sigma} = \alpha\,\ell_{\sigma} + \rho \quad \text{with} \quad 0 \leq \rho < \ell_{\sigma}.
\end{equation}
Set $\varepsilon = 0$ if $\rho = 0$ and $\varepsilon = 1$ otherwise. Then, by~\cite[Proposition~2.3]{AK14}, $\sigma$ admits a \emph{standard refinement} with $\alpha + \epsilon$ maximal subcones. This is obtained by drawing rays through $\alpha +  \varepsilon - 1$ primitive lattice points lying in the strict interior of the line segment $\conv{u,v}$. The elements of $\{\sigma_i\}$, the multiset of maximal subcones in this refinement, are precisely $\alpha$ \emph{primitive }$T$-\emph{cones}, defined by $w_{\sigma_i} = \ell_{\sigma_i}$ ($= \ell_{\sigma}$) and $\varepsilon$ \emph{residual} ($R$-)\emph{cones}, defined by $w_{\sigma_i}$ ($= \rho$) $< \ell_{\sigma_i}$ ($= \ell_{\sigma}$). The primitive $T$-cones are always isomorphic to one another, by~\cite[Proposition 2.3]{AK14}. A standard refinement of $\sigma$ is not unique in general, but the multiset $\{\sigma_i\}$ depends only on $\sigma$. The width and local index of these maximal subcones is determined by the equation~\eqref{eqn:width_divided_by_local_index} and also depends only on $\sigma$. A \emph{standard refinement} $\sref$ of $\spanf{P}$ is a choice of standard refinement for every maximal cone $\sigma \in \spanf{P}$. By construction, the maximal cones of a standard refinement are either primitive $T$-cones or $R$-cones. The polygonal quiver of a Fano polygon will be constructed from a standard refinement of its spanning fan.

\section{Polygonal Quivers}\label{sec:main_construction}

The \emph{set of Fano polygons} $\fanos$ consists of pairs $(N, P)$, where $N$ is an oriented lattice of rank $2$ and $P \subset \NQ$ is a Fano polygon.

\begin{definition}\label{def:quiv_definition}Given $(N,P) \in \mathfrak{F}$, choose a standard refinement $\sref$ of $\spanf{P}$ as in Section~\ref{subsec:background}. The vertex set of $\quiv(N,P)$ is the multiset of inner normal vectors of maximal cones in~$\sref$:
\begin{equation}\label{eqn:vertices_of_quiv}
    \verts{\quiv(N,P)} := \{m_{\sigma} \in M \mid \sigma \text{ is a maximal cone in } \sref\}.
\end{equation}
The number of arrows between $m_{\sigma}$ and $m_{\tau}$ is $\det(m_{\sigma},m_{\tau})$, the coefficient of $m_{\sigma}\wedge m_{\tau}$, pointing from $m_{\sigma}$ to $m_{\tau}$ if the determinant is positive and from $m_{\tau}$ to $m_{\sigma}$ otherwise. Each $m_{\sigma}$ is decorated by $(w_{\sigma},\ell_{\sigma})$.
\end{definition}

Any decorated quiver lying in the image of $\quiv$ is called a \emph{polygonal quiver}.

Definition~\ref{def:quiv_definition} is independent of the choice of standard refinement $\sref$: the set underlying~\eqref{eqn:vertices_of_quiv} is equal to the set of inner normal vectors of maximal cones in $\spanf{P}$; a vector $m$ in this set is the inner normal vector of some maximal cone $\sigma \in \spanf{P}$, and its multiplicity in~\eqref{eqn:vertices_of_quiv} is equal to the total number of primitive $T$ and $R$-cones in a standard subdivision of $\sigma$. The vertex labels are independent of $\sref$ because, as observed in Section~\ref{subsec:background}, they are completely determined by equations of the form~\eqref{eqn:width_divided_by_local_index}, which depend only on the maximal cones $\sigma \in \spanf{P}$.



\begin{example}\label{eg:P1xP1_and_P(1,1,2)}Let $N = \Z^{2}$ with the standard orientation. Consider the Fano polygon $P_1 \subset \NQ = \Q^{2}$ with vertex set $\{(0,-1)$, $(1,0)$, $(0,1)$, $(-1,0)\}$, whose spanning fan $\spanf{P_1}$ defines $\Proj^1 \times \Proj^1$ as a toric variety. The maximal cones of $\spanf{P_1}$ are all smooth and (by convention) are taken to be primitive $T$-cones with $w = \ell = 1$. Thus the standard refinement $\sref_1$ coincides with $\spanf{P_{1}}$. The inner normal vectors of maximal cones in $\sref_1$ are $m_1 = (1,1)^{t}, m_2 = (-1,1)^{t}, m_3 = (-1,-1)^{t}$ and $m_4 = (1,-1)^{t}$ ; $Q_1:=\quiv(\Z^{2},P_1)$ is shown in Figure~\ref{fig:P1xP1_P112_quivers} (top).
\begin{figure}[h!]
    \centering
    \begin{tikzpicture}[thick,scale=0.8, every node/.style={transform shape}]
      \coordinate  (OP1P1) at (-3.5, 2);
      \coordinate  (P1P1v1) at (-3.5, 1);
      \coordinate  (P1P1out1) at (-3.5, 0.5);
      \coordinate  (P1P1v2) at (-2.5, 2);
      \coordinate  (P1P1out2) at (-2, 2);
      \coordinate  (P1P1v3) at (-3.5, 3);
      \coordinate  (P1P1out3) at (-3.5, 3.5);
      \coordinate  (P1P1v4) at (-4.5, 2);
      \coordinate  (P1P1out4) at (-5, 2);
      \coordinate  [label={left:\small $m_1$}]  (P1P1m1) at (-4,1.25);
      \coordinate  [label={right:\small $m_2$}] (P1P1m2) at (-3,1.25);
      \coordinate  [label={right:\small $m_3$}] (P1P1m3) at (-3,2.65);
      \coordinate  [label={left:\small $m_4$}]  (P1P1m4) at (-4,2.65);

      \coordinate  (P1P1arrowh) at (-0.5, 2);
      \coordinate  (P1P1arrowt) at (0.5, 2);

      \coordinate  (OP112) at (-3.5, -2);
      \coordinate  (P112v1) at (-4.5, -2);
      \coordinate  (P112out1) at (-5, -2);
      \coordinate  (P112v2) at (-1.5, -3);
      \coordinate  (P112out2) at (-1, -3.25);
      \coordinate  (P112v3) at (-2.5, -2);
      \coordinate  (P112out3) at (-2, -2);
      \coordinate  (P112v4) at (-3.5, -1);
      \coordinate  (P112out4) at (-3.5, -0.5);
      \coordinate  [label={right:\small $m_1$}] (P112m1) at (-2,-2.35);
      \coordinate  [label={below:\small $m_2$}] (P112m2) at (-3.5,-2.5);
      \coordinate  [label={right:\small $m_3$}] (P112m3) at (-3,-1.35);
      \coordinate  [label={left:\small $m_4$}]  (P112m4) at (-4,-1.35);

      \coordinate  (P112arrowh) at (-0.5, -2);
      \coordinate  (P112arrowt) at (0.5, -2);

      \draw [fill=black] (P1P1v1) circle (0.05);
      \draw [fill=black] (P1P1v2) circle (0.05);
      \draw [fill=black] (P1P1v3) circle (0.05);
      \draw [fill=black] (P1P1v4) circle (0.05);
      \draw [fill=black] (OP1P1) circle (0.05);
      \draw [thick] (P1P1v1) -- (P1P1v2) -- (P1P1v3) -- (P1P1v4) -- (P1P1v1);
      \draw [loosely dotted, thin] (OP1P1) -- (P1P1out1);
      \draw [loosely dotted, thin] (OP1P1) -- (P1P1out2);
      \draw [loosely dotted, thin] (OP1P1) -- (P1P1out3);
      \draw [loosely dotted, thin] (OP1P1) -- (P1P1out4);
      \draw [->,thick] (P1P1arrowh) -- node[above] {$\quiv$} (P1P1arrowt);

      \node at (2.5,3) (quivP1P1m1) {\small$(1,1)$};
      \node at (2.5,1) (quivP1P1m2) {\small$(1,1)$};
      \node at (4.5,1) (quivP1P1m3) {\small$(1,1)$};
      \node at (4.5,3) (quivP1P1m4) {\small$(1,1)$};
      \draw [->, thick] (quivP1P1m1) -- node[left] {$2$} (quivP1P1m2);
      \draw [->, thick] (quivP1P1m2) -- node[below] {$2$} (quivP1P1m3);
      \draw [->, thick] (quivP1P1m3) -- node[right] {$2$} (quivP1P1m4);
      \draw [->, thick] (quivP1P1m4) -- node[above] {$2$}(quivP1P1m1);

      \draw [fill=black] (P112v1) circle (0.05);
      \draw [fill=black] (P112v2) circle (0.05);
      \draw [fill=black] (P112v3) circle (0.05);
      \draw [fill=black] (P112v4) circle (0.05);
      \draw [thick] (P112v1) -- (P112v2) -- (P112v3) -- (P112v4) -- (P112v1);
      \draw [fill=black] (OP112) circle (0.05);
      \draw [loosely dotted, thin] (OP112) -- (P112out1);
      \draw [loosely dotted, thin] (OP112) -- (P112out2);
      \draw [loosely dotted, thin] (OP112) -- (P112out3);
      \draw [loosely dotted, thin] (OP112) -- (P112out4);
      \draw [->,thick] (P112arrowh) -- node[above] {$\quiv$} (P112arrowt);

      \node at (2.5,-1) (quivP112m1) {\small$(1,1)$};
      \node at (2.5,-3) (quivP112m2) {\small$(1,1)$};
      \node at (4.5,-3) (quivP112m3) {\small$(1,1)$};
      \node at (4.5,-1) (quivP112m4) {\small$(1,1)$};
      \draw [->, thick] (quivP112m1) -- node[above] {$2$} (quivP112m4);
      \draw [->, thick] (quivP112m2) -- node[left] {$2$} (quivP112m1);
      \draw [->, thick] (quivP112m2) -- node[below] {$2$} (quivP112m3);
      \draw [->, thick] (quivP112m3) -- node[right] {$2$}(quivP112m4);
      \draw [->, thick] (quivP112m4) -- node[below] {$4$}(quivP112m2);
   \end{tikzpicture}
\caption{The Polygonal Quivers for $\Proj^{1}\times\Proj^{1}$ (top) and $\Proj(1,1,2)$ (bottom).}\label{fig:P1xP1_P112_quivers}
\end{figure}
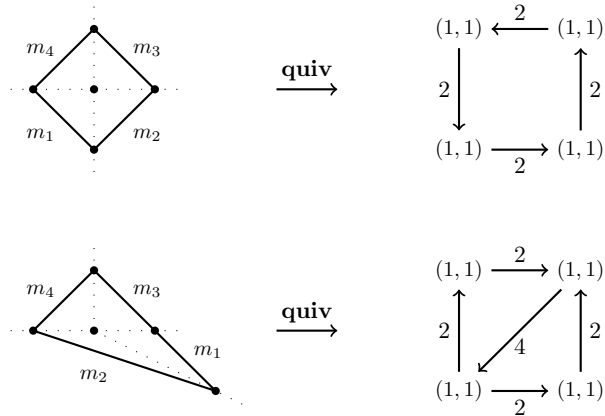

Next, consider $P_2 \in \NQ$ with vertex set $\{(-1,0)$, $(2,-1)$, $(0,1)\}$, which defines $\Proj(1,1,2)$. Consider $\sigma:=\cone{(2,-1),(0,1)} \in \spanf{P_{2}}$ which defines a $\frac{1}{2}(1,1)$ cyclic quotient singularity~\cite[Section~2.2]{Ful93}. Following~\cite[Lemma~2.2]{AK14}, we have $w_{\sigma}= 2$ and $\ell_{\sigma} =1$. Division gives $w_{\sigma} = 2\cdot\ell_{\sigma} + 0$, so $\sigma$ must be subdivided into $2$ primitive $T$-cones via a crepant blow-up through $(1,0)$. The $1$-skeleton of the standard refinement $\sref_2$ consists of rays through $(-1,0)$, $(2,-1)$, $(1,0)$ and $(0,1)$, and the inner normal vectors are $m_1 = (-1,-1)^{t}$, $m_2 = (1,3)^{t}$, $m_3 = (-1,-1)^{t}$ and $m_4 = (1,-1)^{t}$ ; $Q_2:=\quiv(\Z^{2},P_2)$ is shown in Figure~\ref{fig:P1xP1_P112_quivers} (bottom).

Both $P_1,P_2$ are reflexive polygons; notice that every vertex $m$ of $Q_1,Q_2$ satisfies a balancing condition: the number of arrows into $m$ is equal to the number of arrows out of $m$.
\end{example}

\begin{example}The opposite quiver of a polygonal quiver (with the same labels on its vertices) is also polygonal. It arises from the same Fano polygon by reversing the orientation of the ambient lattice.
\end{example}

\begin{example}\label{eg:no_self_loops_or_2-cycles}A polygonal quiver $Q$ never contains self loops or $2$-cycles. This is immediate from Definition~\ref{def:quiv_definition}: the number of arrows from $m \in \verts{Q}$ to itself is $\det(m,m) = 0$, and if $m_1,m_2 \in \verts{Q}$ then (by definition) there are exactly $\det(m_1,m_2)$ arrows between these two vertices. Moreover all of these arrows point in the same direction, determined by the sign of the determinant. The incidence matrix of a polygonal quiver is therefore an exchange matrix for a cluster algebra of geometric type.
\end{example}

\begin{example}\label{eg:normalized_volume_from_quiver}The normalized volume of a Fano polygon $P \subset \NQ$ can be calculated from its polygonal quiver $Q:=\quiv(N,P)$. To see this, observe that:
\begin{equation}\label{eqn:normalized_volume_from_cones}
    \vol(P) = \sum{\vol(\conv{\orig,p,q})},
\end{equation}
where the sum is taken over all maximal cones $\sigma$ (with primitive generators $p,q$) in a standard refinement $\Sigma$ of $\spanf{P}$. Here $\vol$ is the normalized volume, defined to take value $1$ on an empty $2$-simplex in $\NQ$. Now let $\sigma \in \Sigma$ be a maximal cone. There exist coprime positive integers $r,a$ such that $\sigma$ is isomorphic to the cone in $\Z^{2}$ with primitive generators $(0,1),(r,-a)$. Thus $\vol(\conv{\orig,p,q}) = r$. Geometrically, $\sigma$ defines a $\frac{1}{r}(1,a)$ cyclic quotient singularity (see for instance~\cite[Section~2.1]{Ful93}), and it follows from~\cite[Lemma~2.2]{AK14} that $r = w_{\sigma}\ell_{\sigma}$. Thus, by~\eqref{eqn:normalized_volume_from_cones}:
\begin{equation}\label{eqn:normalized_volume_from_quiver}
    \vol(P) = \sum{w_{\sigma}\ell_{\sigma}},
\end{equation}
where the sum is taken over all vertices of $Q$: by Definition~\ref{def:quiv_definition}, these are in one-to-one correspondence with maximal cones $\sigma \in \Sigma$. In particular, the right hand side of~\eqref{eqn:normalized_volume_from_quiver} can be calculated directly from the polygonal quiver $Q$.
\end{example}

\begin{notation}Let $Q$ be a quiver with no $2$-cycles, so that all arrows between any pair of vertices of $Q$ have the same head and tail. Given $m_1,m_2 \in \verts{Q}$, let
\[
    \arr{m_1,m_2 ; Q} = \arr{m_1,m_2} := \mathrm{sgn}(m_1,m_2)\cdot A,
\]
where $A$ is the number of arrows between $m_1$ and $m_2$. The quantity $\mathrm{sgn}(m_1,m_2)$ equals $1$ if all arrows point from $m_1$ to $m_2$ and equals $-1$ otherwise.
\end{notation}

\section{Polygonal Quivers are Balanced}\label{sec:balancing_condition}
The number of arrows between two vertices of a polygonal quiver can be understood in terms of the Fano polygon from which it is constructed. This viewpoint shows that polygonal quivers satisfy a balancing condition at each vertex. In order to correctly formulate this condition, suppose that $Q:=\quiv(N,P)$ is constructed from a standard refinement $\sref$ of $\spanf{P}$. Choose a vertex $m$ of $Q$, at which the balancing condition will be studied. There are two distinguished subsets of vertices of $Q$ determined by $m$:
\begin{align*}
     \outof{m} &:=\{h \in \verts{Q} \mid h \text{ receives an arrow from } m\},\\
     \into{m}  &:=\{t \in \verts{Q} \mid m \text{ receives an arrow from } t\}.
\end{align*}
We may also write $\outof{m;Q}$ and $\into{m;Q}$ if we wish to emphasize the quiver. These sets are disjoint and do not contain $m$ by the observations made in Example~\ref{eg:no_self_loops_or_2-cycles}. They are also nonempty, because every edge of a Fano polygon has at least one non-parallel edge on either side of it. In terms of the Fano polygon $P$, $m \in M$ is the inner normal vector of a maximal cone in $\sref$. In particular, it determines a height function on $\NQ$, so we may take $\hmax$ to be $\sup\{\langle m, x\rangle \mid x \in P\}$ and define $\hmin$ similarly. The \emph{diameter} of $m \in \verts{Q}$ is the integer $D(m):=\hmax - \hmin$. Note that since $P$ contains the origin in its strict interior, we always have $\hmax > 0$ and $\hmin < 0$. Therefore $D(m)$ is always strictly positive. Recalling that every $q \in \verts{Q}$ carries additional data $(w_q, \ell_q) \in \Z^{2}_{\geq 1}$, we have:

\begin{prop}\label{prop:balancing_condition}In the above notation:
\begin{equation}\label{eqn:balancing_condition}
    D(m) = \sum_{h \in \outof{m}}w_h\cdot \arr{m,h} = \sum_{t \in \into{m}}w_t\cdot \arr{t,m}.
\end{equation}
\end{prop}

\begin{proof}Choose an orientation preserving isomorphism between $N$ and $\Z^{2}$ such that $m = (0,1)^{t}$. The situation is illustrated in Figure~\ref{fig:proof_of_balancing_condition}.

\begin{figure}[h!]
    \centering
    \begin{tikzpicture}[thick,scale=0.8, every node/.style={transform shape}]
        \coordinate [label={right:\small\hspace{1mm}$\orig$}] (O) at (0,0);
        \draw [fill=black] (O) circle (0.05);

        \coordinate [label={below:\small $v_0$}] (v0) at (0,-1.5);
        \draw [fill=black] (v0) circle (0.05);
        \coordinate [label={below:\small $v_1$}] (v1) at (1,-1.5);
        \draw [fill=black] (v1) circle (0.05);
        \coordinate [label={right:\small $v_2$}] (v2) at (2,-0.5);
        \draw [fill=black] (v2) circle (0.05);
        \coordinate [label={right:\small $v_s$}] (vs) at (2,0.5);
        \draw [fill=black] (vs) circle (0.05);
        \coordinate [label={above:\small $v_{s+1}$}] (vs1) at (1,1.5);
        \draw [fill=black] (vs1) circle (0.05);
        \coordinate [label={above:\small $v_r$}] (vr) at (-0.5,1.5);
        \draw [fill=black] (vr) circle (0.05);
        \coordinate [label={left:\small $v_{r+1}$}] (vr1) at (-1.5,0.5);
        \draw [fill=black] (vr1) circle (0.05);
        \coordinate [label={left:\small $v_k$}] (vk) at (-1.5,-0.5);
        \draw [fill=black] (vk) circle (0.05);
        \coordinate [label={left:\small $v_{k+1}$}] (vk1) at (-0.5,-1.5);
        \draw [fill=black] (vk1) circle (0.05);

        \draw [thick] (v0) -- (v1) -- (v2);
        \draw [dashed,thin] (v2) -- (vs);
        \draw [thick] (vs) -- (vs1);
        \draw [dashed,thin] (vs1) -- (vr);
        \draw [thick] (vr) -- (vr1);
        \draw [dashed,thin] (vr1) -- (vk);
        \draw [thick] (vk) -- (vk1);
        \draw [dashed,thin] (vk1) -- (v0);

        \coordinate (amt) at (0.5,-1.5);
        \coordinate [label={above:\small $m_0 = m$}] (amh) at (0.5,-1);
        \draw [->,thick] (amt) -- (amh);

        \coordinate (rb) at (4,-1.5);
        \coordinate (rm) at (4,0);
        \coordinate (rt) at (4,1.5);
        \draw [|-|,thick] (rb) -- node[right] {$-\hmin$} (rm);
        \draw [ -|,thick] (rm) -- node[right] {$\hmax$} (rt);

        \coordinate (lb) at (-3.5,-1.5);
        \coordinate (lt) at (-3.5,1.5);
        \draw [|-|,thick] (lb) -- node[left] {$D(m)$} (lt);
    \end{tikzpicture}
\caption{The Proof of Proposition~\ref{prop:balancing_condition}.\label{fig:proof_of_balancing_condition}}
\end{figure}
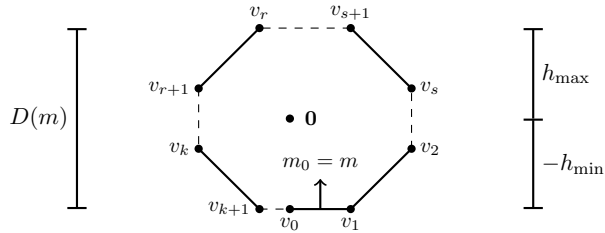
The $v_i$ in Figure~\ref{fig:proof_of_balancing_condition} are primitive lattice vectors determined by the rays of $\sref$. Let $m_i$ denote the inner normal vector of $\sigma_i:=\cone{v_i,v_{i+1}}$, so that $m = m_0$, and write $(w_i,\ell_i)$ for $(w_{\sigma_i},\ell_{\sigma_i})$. In this setup, we have $\outof{m} = \{m_1,\ldots,m_s\}$ and $\into{m} = \{m_r,\ldots,m_k\}$. The first equality follows from the following statement: if $i \in \{1,\ldots,s\}$ then the quantity $w_i\cdot\arr{m,m_i}$ is equal to $\langle m, v_{i+1}\rangle$ - $\langle m, v_i \rangle$~i.e.~to the lattice height, with respect to $m$, of $v_{i+1}$ above $v_i$. To see this, suppose $v_i = (x_i,y_i)$ and $v_{i+1} = (x_{i+1},y_i + y_{i+1})$ in the chosen basis. Note that $y_{i+1} > 0$ by construction. The line segment $\conv{v_i,v_{i+1}}$ has width $w_i$, and hence $m_i = (-y_{i+1}/w_i, (x_{i+1}-x_i)/w_i)^{t}$. Therefore, the quantity $w_i\cdot\arr{m,m_i}$ is equal to:
\begin{equation}\label{eqn:balancing_misc}
    w_i\cdot\det(m,m_i) = \det\mleft(\begin{array}{cc}0 & -y_{i+1}\\1 & x_{i+1}-x_i\end{array}\mright) = y_{i+1},
\end{equation}
and the right hand side of~\eqref{eqn:balancing_misc} is equal to $\langle m, v_{i+1}\rangle$ - $\langle m, v_i \rangle$, as claimed. The second equality now follows from an almost identical argument to the one just given.
\end{proof}

\begin{example}\label{eg:balancing_for_reflexive_polygons}A Fano polygon $P$ is called \emph{reflexive}~\cite{Bat94} if $\ell_{\sigma} = 1$ for all maximal cones $\sigma \in \Sigma_{P}$. The spanning fan of a reflexive polygon admits a unique standard refinement, and the maximal cones of this refinement are all smooth; in particular, their width is equal to $1$. Thus, if $Q$ is the polygonal quiver of a reflexive polygon, then every vertex $m$ of $Q$ carries the label $(1,1)$, and the balancing condition reduces to the statement that the number of arrows into $m$ is equal to the number of arrows out of $m$. See Example~\ref{eg:P1xP1_and_P(1,1,2)}.
\end{example}

\section{Mutation of Balanced Quivers}\label{sec:mutations}
A quiver with no self-loops or $2$-cycles is said to be \emph{decorated} if its vertices carry the additional data of a pair $(w,\ell) \in \Z^{2}$. The \emph{underlying quiver} of a decorated quiver is obtained by forgetting its vertex labels. A decorated quiver $Q$ is \emph{balanced} if every vertex $m$ of $Q$ satisfies the second equality in~\eqref{eqn:balancing_condition}. If $m$ is a vertex of a balanced quiver, then the \emph{diameter} of $m$ is denoted $D(m)$ and is defined by the first equality in~\eqref{eqn:balancing_condition}.

\begin{remark}\label{rem:comparison_of_balanced_and_polygonal_quivers}
Every polygonal quiver is balanced by Definition~\ref{def:quiv_definition} and Proposition~\ref{prop:balancing_condition}. However, for any choice of integer $a \geq 0$ and $\ell_1, \ell_2 \in \Z$, the following balanced quiver is not polygonal:
\begin{figure}[h!]
    \centering
    \begin{tikzpicture}[thick,scale=0.8, every node/.style={transform shape}]
      \node at (-2,0) (v1) {\small$(0,\ell_1)$};
      \node at (2,0)  (v2) {\small$(0,\ell_2)$};

      \draw [->, thick] (v1) -- node[above] {$a$} (v2);
   \end{tikzpicture}
\end{figure}

\noindent This follows from the observation that every polygonal quiver must have at least three vertices. To produce further examples of non-polygonal balanced quivers, note that every vertex label $(w, \ell)$ of a polygonal quiver satisfies some inequalities: $w,\ell \geq 1$, since the width and local index of a cone are positive, and $w \leq \ell$, which follows from the division~\eqref{eqn:width_divided_by_local_index}. See also Example~\ref{eg:family_of_non_polygonal_quivers} for a family of balanced quivers whose vertex labels satisfy both these inequalities, but which are not polygonal.
\end{remark}

Mutation at a vertex~\cite{Sei94,FZ1} is an operation defined on the underlying quiver of any decorated quiver. We begin by extending this operation to balanced quivers, by describing an accompanying transformation of the vertex labels.

\begin{definition}\label{def:mutation_of_balanced_quivers}Let $Q$ be a balanced quiver. Choose a vertex $m$ of $Q$ with label $(w, \ell) \in \Z^{2}$. The mutation of $Q$ at $m$ is denoted $\mut_{m}(Q)$ and is the decorated quiver whose underlying quiver is the (usual) mutation of the underlying quiver of $Q$ at $m$. A vertex $v \neq m$ of $\mut_{m}(Q)$ is decorated with the same label as in $Q$ and the vertex $m$ in $\mut_{m}(Q)$ is decorated with the label $(D(m)- w, D(m) - \ell)$.
\end{definition}

\begin{example}
The quiver in Figure~\ref{fig:P1xP1_P112_quivers} (bottom) is obtained by mutating the quiver in Figure~\ref{fig:P1xP1_P112_quivers} (top) at the top-left vertex, which has diameter $2$. Notice that the Fano polygon in Figure~\ref{fig:P1xP1_P112_quivers} (bottom) is a combinatorial mutation~\cite{ACGK12} of the the Fano polygon in Figure~\ref{fig:P1xP1_P112_quivers} (top) with respect to the width vector $m_1$ and a factor of unit length. Thus, in this example, constructing combinatorial mutations commutes with the map $\quiv$. See Proposition~\ref{prop:quiv_mutation_commutes_with_comb_mutation_at_Tcones}.
\end{example}

\begin{prop}\label{prop:balanced_quivers_are_closed_under_mutation}If $Q$ is a balanced quiver then $\mut_{m}(Q)$ is balanced for any $m \in \verts{Q}$. Thus the class of balanced quivers is closed under mutation.
\end{prop}

\begin{proof}For convenience, set $Q' = \mut_{m}(Q)$. Then $Q'$ is balanced at $m$ by the following three observations: $Q$ is balanced at $m$, the label of any vertex $v \neq m$ of $Q'$ is the same as in $Q$ and mutation of the underlying quiver at $m$ merely reverses the arrows incident at $m$. It remains to show that $Q'$ is balanced at every vertex $v \neq m$. Note that balancing at $v$ is equivalent to the vanishing of the following quantity:
\begin{equation}\label{eqn:sum_formula_for_balancing_condition}
    S(v;Q') := \sum_{y \in \verts{Q'}}w_y\cdot\arr{v,y;Q'},
\end{equation}
where the vertex $y$ carries the label $(w_y,\ell_y)$. Choose a vertex $v_0 \neq m$ of $Q'$. If $v_0$ is not connected to $m$ in $Q$ then the arrows incident at $v_0$ will not change under mutation at $m$, and the labels of vertices connected to $v_0$ will also not change. Since $Q$ is balanced at $v_0$, it then follows that $Q'$ is balanced at $v_0$. Suppose now that $v_0$ is connected to $m$ in $Q$. Then $v_0$ lies in exactly one of the sets $\into{m;Q}, \outof{m;Q}$, defined in Section~\ref{sec:balancing_condition}. Suppose $v_0 \in \into{m;Q}$, as the other case is almost identical.

Let $v_1,\ldots,v_n$ be the vertices of $Q$ different from $v_0$ and $m$, labelled such that $\outof{m;Q} = \{v_1,\ldots,v_k\}$ for some integer $k$ satisfying $0 < k \leq n$. Let the label of $m$ in $Q$ be $(w,\ell)$, and the label of $v_i$ in $Q$ be $(w_i,\ell_i)$ for $i = 0,\ldots,n$. Given vertices $x, y$ we simplify notation by setting $A(x,y) = \arr{x,y;Q'}$ and $\arr{x,y} = \arr{x,y;Q}$. Then by definition of quiver mutation, the arrows incident at $v_0$ in $Q'$ can be enumerated as follows:
\begin{itemize}
    \item
    $A(v_0,m) = -\arr{v_0,m}$;
    \item
    $A(v_0,y) = \arr{v_0,m}\arr{m,y} + \arr{v_0,y}$, if $y \in \outof{m}$;
    \item
    $A(v_0,y) = \arr{v_0,y}$ otherwise.
\end{itemize}
Substituting these quantities into Equation~\eqref{eqn:sum_formula_for_balancing_condition}, the quantity $S(v_0;Q)$ equals:
\[
    -(D(m) - w)\cdot\arr{v_0,m} + \sum_{j = 1}^{k}{w_j}\cdot\mleft(\arr{v_0,m}\arr{m,v_j} + \arr{v_0,v_j}\mright) + \sum_{j = k+1}^{n}w_j\cdot\arr{v_0,v_j},
\]
were $D(m)$ is the diameter of $m$ in Q. This can be rearranged as follows:
\[
    \mleft[w\cdot\arr{v_0,m} + \sum_{j = 1}^{n}{w_j\cdot\arr{v_0,v_j}}\mright] + \arr{v_0,m}\mleft[-D(m) + \sum_{j = 1}^{k}w_j\cdot\arr{m,v_j}\mright].
\]
The quantity inside the left bracket equals $S(m;Q)$, and equals zero because $Q$ is balanced at $m$. The quantity inside the right bracket is zero by the definition of $D(m)$ in $Q$. Thus, $S(v_0;Q') = 0$ and $Q'$ is balanced at $v_0$. We conclude that $Q' = \mut_{m}(Q)$ is balanced.
\end{proof}

\begin{example}\label{eg:polygonal_quivers_are_not_closed_under_mutation}Let $N = \Z^{2}$ with the standard orientation. The spanning fan of the Fano polygon $P \subset \Q^{2}$ with vertex set $\{(1,0),(0,1),(-1,-6)\}$ defines $\Proj(1,1,6)$ as a toric variety, and equals its own standard refinement. The associated polygonal quiver is shown on the left of Figure~\ref{fig:polygonal_quivers_are_not_closed_under_mutation}.
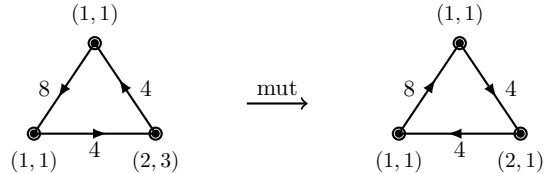
\begin{figure}[h!]
    \centering
    \begin{tikzpicture}[thick,scale=0.8, every node/.style={transform shape}]
      \coordinate  (Q1v1) at (-2, -1);
      \coordinate [label={below:\small $(2,3)$}]  (P2v1_lab) at (-2, -1.15);
      \coordinate  (Q1v2) at (-4, -1);
      \coordinate [label={below:\small $(1,1)$}]  (P2v2_lab) at (-4, -1.15);
      \coordinate  (Q1v3) at (-3, 0.5);
      \coordinate [label={above:\small $(1,1)$}]  (P2v3_lab) at (-3, 0.65);
      \draw [fill=black] (Q1v1) circle (0.05);
      \draw (Q1v1) circle (0.1);
      \draw [fill=black] (Q1v2) circle (0.05);
      \draw (Q1v2) circle (0.1);
      \draw [fill=black] (Q1v3) circle (0.05);
      \draw (Q1v3) circle (0.1);
      \draw [middlearrow={latex}] (Q1v1) -- node[right] {\hspace{1mm}$4$} (Q1v3);
      \draw [middlearrow={latex}] (Q1v3) -- node[left] {\hspace{-6mm}$8$} (Q1v2);
      \draw [middlearrow={latex}] (Q1v2) -- node[below] {$4$} (Q1v1);

      \coordinate (arrowt) at (-0.5,-0.5);
      \coordinate (arrowh) at (0.5,-0.5);
      \draw [->,thick] (arrowt) -- node[above] {$\mut$} (arrowh);

      \coordinate  (Q2v1) at (4, -1);
      \coordinate [label={below:\small $(2,1)$}]  (Q2v1_lab) at (4, -1.15);
      \coordinate  (Q2v2) at (2, -1);
      \coordinate [label={below:\small $(1,1)$}]  (Q2v2_lab) at (2, -1.15);
      \coordinate  (Q2v3) at (3, 0.5);
      \coordinate [label={above:\small $(1,1)$}]  (Q2v3_lab) at (3, 0.65);
      \draw [fill=black] (Q2v1) circle (0.05);
      \draw (Q2v1) circle (0.1);
      \draw [fill=black] (Q2v2) circle (0.05);
      \draw (Q2v2) circle (0.1);
      \draw [fill=black] (Q2v3) circle (0.05);
      \draw (Q2v3) circle (0.1);
      \draw [middlearrow={latex}] (Q2v3) -- node[right] {$\hspace{1mm}4$} (Q2v1);
      \draw [middlearrow={latex}] (Q2v1) -- node[below] {$4$} (Q2v2);
      \draw [middlearrow={latex}] (Q2v2) -- node[left] {\hspace{-6mm}$8$} (Q2v3);
   \end{tikzpicture}
\caption{The class of Polygonal Quivers is not closed under Mutation.\label{fig:polygonal_quivers_are_not_closed_under_mutation}}
\end{figure}
The vertex with label $(2,3)$ has diameter $4$, and mutation at this vertex produces the balanced quiver shown on the right of Figure~\ref{fig:polygonal_quivers_are_not_closed_under_mutation}. The mutated quiver is not polygonal, since one of it's vertices carries the label $(w,\ell) = (2,1)$, which violates the inequality: $w \leq \ell$. Thus, the class of polygonal quivers is not closed under mutation.
\end{example}

In light of Example~\ref{eg:polygonal_quivers_are_not_closed_under_mutation}, it is natural to ask the extent to which polygonal quivers fail to be closed under mutation. The vertices of a polygonal quiver $Q:=\quiv(N,P)$ can be partitioned into two types: \emph{primitive} $T$-\emph{vertices} are those which are inner normal vectors to primitive $T$-cones; they are those vertices of $Q$ whose label $(w,\ell)$ satisfies $w = \ell$. $R$-\emph{vertices} are inner normals to $R$-cones and their labels $(w,\ell)$ satisfy $w < \ell$.

\begin{prop}\label{prop:quiv_mutation_commutes_with_comb_mutation_at_Tcones}Let $m$ be a vertex of the polygonal quiver $Q:=\quiv(N,P)$. The quiver $\mut_m(Q)$ is polygonal if and only if $m$ is a primitive $T$-vertex, and in this case:
\begin{equation}\label{eqn:mutation_formula_polygonal_quivers}
    \mut_m(Q) \cong \quiv(N,\mut_m(P)),
\end{equation}
where $\mut_m(P)$ is the combinatorial mutation of $P$ with respect to the width vector $m$ and a factor of unit length, as defined in~\cite{ACGK12}.
\end{prop}

\begin{proof}If $m$ is not a primitive $T$-vertex then it must be an $R$-vertex, and its label $(w,\ell)$ in $Q$ must satisfy $w <  \ell$. But then the label of $m$ in $\mut_m(Q)$ is $(w',\ell') := (D(m) - w, D(m) - \ell)$, and this satisfies $w' > \ell'$, showing that $\mut_m(Q)$ is not polygonal. Conversely, let $m$ be a primitive $T$-vertex. Note that $\mut_m(P)$ is a Fano polygon by~\cite[Proposition~2]{ACGK12}, so $\quiv(N,\mut_m(P))$ is a well-defined polygonal quiver. The isomorphism~\eqref{eqn:mutation_formula_polygonal_quivers} holds on the level of underlying quivers by an argument identical to the one given in~\cite[Proposition~3.17]{KNP15}. It remains to show that this isomorphism preserves vertex labels.

First consider $m$ as a vertex of $Q$. Here, $m$ is the inner normal vector of a primitive $T$-cone in $P$ whose width and local index both equal $-\hmin$ (as defined in Section~\ref{sec:balancing_condition}). Thus the label of $m$ in $Q$ is $(-\hmin,-\hmin)$, which implies that the label of $m$ in $\mut_m(Q)$ is
\[
 (D(m)+\hmin,D(m)+\hmin) = (\hmax,\hmax).
\]
Now $m$ corresponds to some vertex $m'$ of $\quiv(N,\mut_m(P))$, by the isomorphism of underlying quivers. By the definition of this isomorphism, $m'$ is the inner normal vector to a primitive $T$-cone in $\mut_m(P)$ whose width and local index both equal $\hmax$. Thus the label of $m'$ is $(\hmax,\hmax)$, which coincides with that of $m$.

Finally consider a vertex $v \neq m$ of $\mut_m(Q)$. As a vertex of $Q$, $v$ is the inner normal vector of a primitive $T$ or $R$-cone $\sigma$ in $P$. The label of $v$ in $Q$, and hence in $\mut_m(Q)$, is the width and local index of $\sigma$. The vertex $v'$ of $\quiv(N,\mut_m(P))$ which corresponds to $v$ under the isomorphism of underlying quivers, is the inner normal vector of a cone $\sigma'$ which is isomorphic to $\sigma$ (see~\cite[Proposition~3.6]{AK14}). Thus the label of $v'$ is also the width and local index of $\sigma$. We conclude that the isomorphism of underlying quivers preserves vertex labels, as claimed.
\end{proof}
The quiver of Example~\ref{eg:polygonal_quivers_are_not_closed_under_mutation} was mutated at an $R$-vertex and did not remain polygonal, as expected. Quiver mutation at an $R$-vertex does not appear to have an analogue in terms of Fano polygons since $R$-cones are, by definition, rigid under combinatorial mutations.

\begin{remark}More generally, one may consider a quiver $Q$, with no self loops or $2$-cycles, whose vertices $v$ are labelled by elements $w_v$ of some abelian group $A$. Such a $Q$ is \emph{balanced} at a vertex $m$ if the following equality holds:
\begin{equation}\label{eqn:balancing_condition_abelian}
  \sum_{h \in \outof{m}}w_h\cdot \arr{m,h ; Q} = \sum_{t \in \into{m}}w_t\cdot \arr{t,m ; Q}.
\end{equation}
Every balanced vertex $m$ of $Q$ has \emph{diameter} $D(m) \in A$, given by either side of \eqref{eqn:balancing_condition_abelian}. We can define \emph{mutation} of $Q$ at a balanced vertex $m$ exactly as in Definition~\ref{def:mutation_of_balanced_quivers}. In particular:

\begin{prop}[Mutation at a Balanced Vertex]
If $Q$ is balanced at $m$ then every balanced vertex of $Q$ remains balanced in $\mut_m(Q)$.
\end{prop}

The proof is identical to that of Proposition~\ref{prop:balanced_quivers_are_closed_under_mutation}. Since mutation is an involution, we have:

\begin{cor}
The number of balanced vertices remains constant whenever $Q$ is mutated at a balanced vertex.
\end{cor}
\end{remark}

\begin{remark}Let $Q$ be an unlabelled quiver with vertex set $m_1,\ldots,m_n$, allowing the possibility of self-loops and $2$-cycles. The $n \times n$ exchange matrix of $Q$ is $E:=(\arr{m_i,m_j})$. Fix an abelian group $A$, and label the vertex $m_i$ of $Q$ by an element $w_i$ of $A$. This gives a vector $w:=(w_1,\ldots,w_n) \in A^{n}$. Observe that $Q$ is balanced at $m_i$ with respect to the labelling $w$ (i.e.~the analogue of~\eqref{eqn:sum_formula_for_balancing_condition} vanishes) if and only if the $j^{\mathrm{th}}$ entry of $E\cdot w$ is zero. In particular, $\ker(E) \subset A^{n}$ is the space of all vertex labellings (over $A$) for which $Q$ is balanced.
\end{remark}

Let $Q$ be a quiver with no self-loops or $2$-cycles, which may or may not be decorated. Consider the positive integer $\ind{Q}:=\gcd{\arr{x,y} \mid x,y \in \verts{Q}}$. Then:

\begin{prop}\label{prop:arrow_invariant}For any vertex $m$ of $Q$ we have: $\ind{Q} = \ind{\mut_{m}(Q)}$.
\end{prop}

\begin{proof}Given two vertices $x,y$, we simplify notation by setting $A(x,y) = \arr{x,y ; \mut_m(Q)}$ and $\arr{x,y} = \arr{x,y ; Q}$. By the definition of (usual) quiver mutation we have:
\begin{itemize}
    \item
    $A(m,y) = -\arr{m,y}$ for all $y \in \into{m} \cup \outof{m}$;
    \item
    $A(x,y) = \arr{x,m}\arr{m,y} + \arr{x,y}$ whenever $x,y \neq m$.
\end{itemize}
Since $\ind{Q} = \gcd{\arr{x,y}}$ and $\ind{\mut_{m}(Q)} = \gcd{A(x,y)}$, it follows immediately that $\ind{Q}$ divides $\ind{\mut_{m}(Q)}$ and vice versa. Both integers are positive, so they must be equal.
\end{proof}

 Suppose further that the vertices $m$ of $Q$ are decorated with integer labels $w_m$ and that $Q$ is balanced with respect to these labels (so that the diameter of each vertex is defined). Let $w(Q)$ be the positive integer $\gcd{ w_v \mid v \in \verts{Q}}$.

\begin{prop}\label{prop:weight_invariant}For any vertex $m$ of $Q$ we have $w(Q) = w(\mut_{m}(Q))$.
\end{prop}

\begin{proof}Let the vertex labels of $Q$ be $w_m,w_1,\ldots,w_k$. By Definition~\ref{def:mutation_of_balanced_quivers}, the vertex labels of $\mut_{m}(Q)$ are $D(m)-w_m,w_1,\ldots,w_k$, where $D(m)$ is a weighted sum of $w_1,\ldots,w_k$. Therefore $w(Q)$ divides $w(\mut_m(Q))$ and vice versa.
\end{proof}

\section{The Quiver Degree Formula for Polygonal Quivers}\label{sec:quiver_degree_formula}

\subsection{Anticanonical Degree}\label{subsec:degree_derivation}
In our proof of the balancing condition (Proposition~\ref{prop:balancing_condition}), the number of arrows between two vertices of a polygonal quiver $\quiv(N,P)$ was interpreted in terms of lattice distances in $\NQ$. There is another interpretation, in terms of volumes, which is more natural from the perspective of toric geometry. This viewpoint allows the anticanonical degree of toric Fano surfaces to be computed in terms of the associated polygonal quivers.

\vspace{3mm}
Fix an oriented lattice $N$ of rank $2$ and a Fano polygon $P \subset \NQ$. Since $P$ contains the origin in its strict interior, there is a toric surface $X$, constructed from the spanning fan $\spanf{P}$, and an ample line bundle on $X$, whose space of sections has a basis indexed by lattice points of the dual polygon $\dual{P} \subset \MQ$ (see~\cite{Ful93}). The Fano property (primitivity of vertices) implies that this ample line bundle is the anticanonical: $-K_{X}$. Therefore:
\[
    (-K_X)^{2} = \vol(\dual{P}),
\]
where $\vol$ is the normalized volume, taking value $1$ on an empty $2$-simplex in $\MQ$. Fix an isomorphism of oriented lattices between $N$ and $\Z^2$. Label the vertices of $\dual{P}$ cyclically: $v_1,\ldots,v_n$, in a way that agrees with the orientation. Then $\det(v_k,v_{k+1})$ is the normalized volume of the $2$-simplex $\conv{\orig,v_k,v_{k+1}}$ in $\MQ$, and it follows that
\begin{equation}\label{eqn:degree_determinant_form}
    (-K_X)^2 = \sum_{k=1}^{n}\det(v_k,v_{k+1}).
\end{equation}
The right side of~\eqref{eqn:degree_determinant_form} can be computed in terms of $Q:=\quiv(N,P)$, as follows: the cyclic numbering of the vertices of $\dual{P}$ is the same as a cyclic numbering $E_1,\ldots,E_n$ of the edges of $P$. For each $E_i$ we may choose an element $m_i$ of the (nonempty) multiset:
\begin{equation}\label{eqn:vertex_multiset_in_degree_derivation}
    \{m \in \verts{Q} \mid m \text{ is an inner normal vector of the cone over }E_i\}.
\end{equation}
The precise choice of $m_i$ is unimportant because in general~\eqref{eqn:vertex_multiset_in_degree_derivation} contains multiple copies of each inner normal vector, distinguished only by vertex labels. By Definition~\ref{def:quiv_definition}, each $m_i$ is labelled by a pair $(w_i,\ell_i) \in \Z^2_{\geq 1}$ and it is immediate from the definition of dual polygons that
\[
v_i = (1/\ell_i)m_i \text{ for }i=1,\ldots,n.
\]
Substituting these expressions for the $v_i$ into~\eqref{eqn:degree_determinant_form}, and recalling that $\det(m_k,m_{k+1})$ equals $\arr{m_k,m_{k+1}}$, we arrive at the following formula:
\begin{equation}\label{eqn:degree_quiver_form}
    (-K_X)^2 = \sum_{k = 1}^{n} \frac{\arr{m_k,m_{k+1}}}{\ell_k\ell_{k+1}}.
\end{equation}
In particular, the right side of~\eqref{eqn:degree_quiver_form} can be computed directly from the polygonal quiver $Q$.

\begin{remark}[Cyclic Subquivers]\label{rmk:cyclic_subquivers}Notice that the method used above to compute $(-K_X)^2$ can be re-interpreted as a recipe for constructing a special class of cyclic subquivers of a polygonal quiver. In the notation of Section~\ref{subsec:degree_derivation}, there is one cyclic subquiver $C$ of $Q$ for each choice of ordered list $m_1, \ldots, m_n$: we set $\verts{C} = \{m_1,\ldots,m_n\}$ and $\arr{m_1,m_j;C} := \arr{m_i,m_j;Q}$ if $j = (i+1)\mod{n}$  and $\arr{m_i,m_j;C}:=0$ otherwise. Every one of these subquivers has $|\verts{P}|$ vertices and determines $(-K_X)^2$ via the formula~\eqref{eqn:degree_quiver_form}. Moreover, since at least one such cyclic subquiver always exists, we see in particular that every polygonal quiver contains at least one oriented cycle.
\end{remark}

\begin{example}Let $N = \Z^{2}$ with the standard orientation. Consider the Fano polygon $P$ of $X_1 = \Proj(1,1,2)$, and its polygonal quiver $Q$, as shown in Figure~\ref{fig:P1xP1_P112_quivers} (bottom). To compute $(-K_{X_1})^2$ from $Q$, begin by labelling the edges of $P$ in a counter-clockwise manner (i.e.~in a manner consistent with the chosen orientation). This labelling determines an ordered list of $n = |\verts{P}| = 3$ vertices of $Q$: one inner normal is chosen for each edge. The only non-unique choice is for the edge $\conv{(2,-1),(0,1)}$, where we may choose either $m_1$ or $m_3$ (in the notation of Figure~\ref{fig:P1xP1_P112_quivers}, bottom). Our choice determines a cyclic subquiver of $Q$ with $|\verts{P}| = 3$ vertices. The two possibilities for this subquiver are shown in Figure~\ref{fig:cyclic_subquivers_P112}.

\begin{figure}[h!]
    \centering
    \begin{tikzpicture}[thick,scale=0.8, every node/.style={transform shape}]
      \node at (0.5,-1) (C1v1) {\small$(1,1)$};
      \node at (2.5,-1) (C1v2) {\small$(1,1)$};
      \node at (2.5,1) (C1v3) {\small$(1,1)$};
      \draw [->,thick] (C1v3) -- node[left] {$4$\hspace{2mm}{ }} (C1v1);
      \draw [->, thick] (C1v1) -- node[below] {$2$} (C1v2);
      \draw [->, thick] (C1v2) -- node[right] {$2$}(C1v3);

      \node at (-2.5,-1) (C2v1) {\small$(1,1)$};
      \node at (-2.5,1) (C2v2) {\small$(1,1)$};
      \node at (-0.5,1) (C2v3) {\small$(1,1)$};

      \draw [->,thick] (C2v1) -- node[left] {$2$} (C2v2);
      \draw [->, thick] (C2v2) -- node[above] {$2$} (C2v3);
      \draw [->, thick] (C2v3) -- node[right] {\hspace{1mm}{ }$4$} (C2v1);
   \end{tikzpicture}
\caption{Two distinguished subquivers for $\Proj(1,1,2)$.}\label{fig:cyclic_subquivers_P112}
\end{figure}
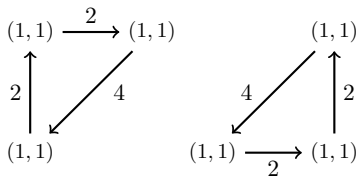
The anticanonical degree of $\Proj(1,1,2)$ can now be read off from either of these subquivers: in both cases we have $(-K_{X_1})^{2} = 2/(1\cdot1) + 4/(1\cdot1) + 2/(1\cdot1) = 8$, as expected.

For another example, consider $X_2 = \Proj(1,1,6)$ (cf.~Example~\ref{eg:polygonal_quivers_are_not_closed_under_mutation}). The associated polygonal quiver $Q$ has three vertices, and is shown in Figure~\ref{fig:polygonal_quivers_are_not_closed_under_mutation}. Since the Fano polygon of $\Proj(1,1,6)$ also has three vertices, there is only one cyclic subquiver and it equals $Q$. The formula~\eqref{eqn:degree_quiver_form} then tells us that $(-K_{X_{2}})^{2} = 8/(1\cdot1) + 4/(1\cdot3) + 4/(1\cdot3) = 32/3$, as expected.
\end{example}

\subsection{Quiver Degree Formula}\label{subsec:riemann_roch}
Let $P \subset \NQ$ be a Fano polygon, whose spanning fan defines the toric Fano surface $X_P$. We recall from~\cite{AK14}, that the \emph{singularity content} of $P$ is a pair $(\tau, \basket)$ where $\tau$ is a non-negative integer and $\basket$ is a multiset of $R$-singularities (or equivalently: of $R$-cones, in the sense of Section~\ref{subsec:background}). Singularity content is an invariant of combinatorial mutation, and it determines the anticanonical degree of $X_P$ via the following Noether formula:
\begin{equation}\label{eqn:degree_from_singularity_content}
    (-K_{X_P})^2 = 12 - \tau + \sum{A(\sigma)},
\end{equation}
where the sum is taken over all $\sigma \in \basket$. The rational numbers $A(\sigma)$ are defined in~\cite{AK14}, and can be computed explicitly for any $R$-cone. Now if $Q = \quiv(N,P)$, then we may use the construction of Section~\ref{subsec:degree_derivation} to obtain a cyclic subquiver $C$ of $Q$, which also determines $(-K_{X_P})^2$. Combining the formulae~\eqref{eqn:degree_quiver_form} and \eqref{eqn:degree_from_singularity_content}, we obtain the \emph{quiver degree formula}, relating the combinatorics of $Q$ to the singularities of $X_P$:
\begin{equation}\label{eqn:quiver_degree_formula}
    \sum_{k = 1}^{n} \frac{\arr{m_k,m_{k+1}}}{\ell_k\ell_{k+1}} = 12 - \tau + \sum_{\sigma \in \basket}{A(\sigma)}.
\end{equation}
Here, $m_1, \ldots, m_n$ are the (cyclically ordered) vertices of $C$, $m_i$ has label $(m_i, \ell_i)$, $i = 1,\ldots,n$ and $\arr{m_k,m_{k+1}} = \arr{m_k,m_{k+1};Q}$. As discussed in Section~\ref{subsec:degree_derivation}, the left hand side of~\eqref{eqn:quiver_degree_formula} is independent of the choice of cyclic subquiver $C$.

\section{Block Quivers of Decorated Quivers}\label{sec:block_quivers}
We have seen in Section~\ref{sec:mutations} that polygonal quivers behave well if one is primarily interested in combinatorial mutations of the underlying Fano polygons. However, since they are constructed from standard refinements of spanning fans, polygonal quivers often contain repeated information, which is inconvenient from the viewpoint of calculations. In this situation, it is often useful to pass to the \emph{block quiver}. In particular, the block construction removes the need for choices in Section~\ref{subsec:degree_derivation}, so that every block quiver of a Fano polygon has a distinguished \emph{Hamiltonian subquiver}.

\begin{definition}\label{def:block_quiver_of_decorated_quiver}The block quiver of a decorated quiver $Q$ is denoted $Q_b$. It has vertex set $\verts{Q}/{\sim}$ where $m \sim m'$ if and only if the following conditions hold:
\begin{enumerate}[label=(\alph*),noitemsep]
  \item\label{misc:01}
  The labels $(w,\ell),(w',\ell')$ of $m,m'$ satisfy $\ell = \ell'$.
  \item\label{misc:02}
  $\into{m} = \into{m'}$ and $\arr{x,m} = \arr{x,m'}$ for all $x \in \into{m}$.
  \item\label{misc:03}
  $\outof{m} = \outof{m'}$ and $\arr{x,m} = \arr{x,m'}$ for all $x \in \outof{m}$.
\end{enumerate}
A vertex $v$ of $Q_b$ is thus an equivalence class of vertices of $Q$: $v = \{m_1,\ldots,m_k\}$. The label of $v$ is defined to be $(w_1+\ldots+w_k, \ell)$, where $(w_i,\ell)$ is the label of $m_i$ for $i = 1,\ldots,k$. Given two vertices $v_1,v_2$ of $Q_b$, we define $\arr{v_1,v_2;Q_b}$ to be $\arr{m_1,m_2;Q}$, where $m_1,m_2$ are representatives of $v_1,v_2$ respectively.
\end{definition}

A \emph{block quiver} is any decorated quiver $Q$ satisfying $Q = Q_b$. Note that the quantity $\arr{v_1,v_2;Q_b}$ in Definition~\ref{def:block_quiver_of_decorated_quiver} is independent of the choice of representatives: suppose $m_1 \sim m_1'$ and $m_2 \sim m_2'$. First note that $\arr{m_1',m_2';Q} = \arr{m_1,m_2';Q}$. This follows from the conditions~\ref{misc:02} and~\ref{misc:03} which state that $m_2' \in \into{m_1}\cup\outof{m_1}$ if and only if $m_2' \in \into{m_1'}\cup\outof{m_1'}$, and in this case $\arr{m_1,m_2';Q} = \arr{m_1',m_2';Q}$. Otherwise, we have $\arr{m_1,m_2';Q} = 0 = \arr{m_1',m_2';Q}$. A similar argument also shows that $\arr{m_1,m_2;Q} = \arr{m_1,m_2';Q}$.

\begin{lemma}\label{lem:block_quivers_also_balanced}If $Q$ is a balanced quiver then its block quiver $Q_b$ is also balanced.
\end{lemma}

\begin{example}\label{eg:block_quiver_P1XP1_P112}The quiver shown in Figure~\ref{fig:P1xP1_P112_quivers} (top) is a block quiver. The quiver shown in Figure~\ref{fig:P1xP1_P112_quivers} (bottom) has balanced block quiver shown on the right of Figure~\ref{fig:passage_to_block_quiver}.
\begin{figure}[h!]
    \centering
    \begin{tikzpicture}[thick,scale=0.8, every node/.style={transform shape}]
      \node at (-4,1) (quivP112m1) {\small$(1,1)$};
      \node at (-4,-1) (quivP112m2) {\small$(1,1)$};
      \node at (-2,-1) (quivP112m3) {\small$(1,1)$};
      \node at (-2,1) (quivP112m4) {\small$(1,1)$};
      \draw [->, thick] (quivP112m1) -- node[above] {$2$} (quivP112m4);
      \draw [->, thick] (quivP112m2) -- node[left] {$2$} (quivP112m1);
      \draw [->, thick] (quivP112m2) -- node[below] {$2$} (quivP112m3);
      \draw [->, thick] (quivP112m3) -- node[right] {$2$}(quivP112m4);
      \draw [->, thick] (quivP112m4) -- node[below] {$4$}(quivP112m2);

      \coordinate (arrowt) at (-0.5,0);
      \coordinate (arrowh) at (1,0);
      \draw [->,thick] (arrowt) -- node[above] {\textbf{block}} (arrowh);

      \node at (2,-1) (Bv1) {\small$(1,1)$};
      \node at (4,-1) (Bv2) {\small$(2,1)$};
      \node at (4,1) (Bv3) {\small$(1,1)$};
      \draw [->,thick] (Bv3) -- node[left] {$4$\hspace{2mm}{ }} (Bv1);
      \draw [->, thick] (Bv1) -- node[below] {$2$} (Bv2);
      \draw [->, thick] (Bv2) -- node[right] {$2$}(Bv3);
   \end{tikzpicture}
   \caption{Passage to the Block Quiver.\label{fig:passage_to_block_quiver}}
\end{figure}
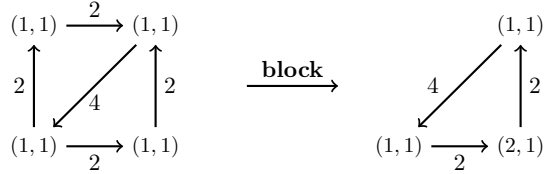
Informally speaking, the block quiver has been constructed from the original by glueing the top-left and bottom-right vertices, which have the same local structure in terms of incident arrows, and adding their weight labels. We emphasize that this block quiver is different from the cyclic subquivers shown in Figure~\ref{fig:cyclic_subquivers_P112}. In general the two constructions are different, and a block quiver need not be cyclic.
\end{example}

The block quiver of a polygonal quiver can be computed directly from its Fano polygon.

\begin{definition}\label{def:bquiv_definition}Let $(N,P) \in \mathfrak{F}$. The vertex set of the quiver $\bquiv(N,P)$ is the set of inner normal vectors of maximal cones in the spanning fan~$\spanf{P}$:
\begin{equation}\label{eqn:vertices_of_bquiv}
    \verts{\bquiv(N,P)} := \{m_{\sigma} \in M \mid \sigma \text{ is a maximal cone in } \spanf{P}\}.
\end{equation}
The number of arrows between $m_{\sigma}$ and $m_{\tau}$ is $\det(m_{\sigma},m_{\tau})$, the coefficient of $m_{\sigma}\wedge m_{\tau}$, pointing from $m_{\sigma}$ to $m_{\tau}$ if the determinant is positive and from $m_{\tau}$ to $m_{\sigma}$ otherwise. Each $m_{\sigma}$ is decorated by $(w_{\sigma},\ell_{\sigma})$.
\end{definition}

To show that this construction coincides with the one given in Definition~\ref{def:block_quiver_of_decorated_quiver} for polygonal quivers, consider vertices $m_1,m_2$ of $\quiv(N,P)$. Then $m_1, m_2$ are, by definition, inner normal vectors of maximal cones in a standard refinement of $\spanf{P}$. In particular, $m_1$ and $m_2$ are inner normal vectors of (uniquely determined) maximal cones $\sigma_1$ and $\sigma_2$ in $\spanf{P}$.

\begin{lemma}\label{lem:block_quiver_definitions_coincide_for_polygonal_quivers}In the above notation: $m_1 \sim m_2$ if and only if $\sigma_1 = \sigma_2$. Thus
\[
    \quiv(N,P)_b \cong \bquiv(N,P),
\]
that is, the construction of Definition~\ref{def:bquiv_definition} produces the block quiver of $\quiv(N,P)$.
\end{lemma}

\begin{proof}If $\sigma_1 = \sigma_2$ then $m_1 = m_2$, and hence $m_1 \sim m_2$. Conversely, if $\sigma_1 \neq \sigma_2$ then there are two possibilities: first assume $m_1 \in \into{m_2} \cup \outof{m_2}$. If $m_1 \in \into{m_2}$ then $\into{m_1} \neq \into{m_2}$, because polygonal quivers have no self-loops (Example~\ref{eg:no_self_loops_or_2-cycles}). Thus, $m_1 \not\sim m_2$. A similar argument shows $m_1 \not\sim m_2$ when $m_1 \in \outof{m_2}$. Next assume $m_1 \not\in \into{m_2} \cup \outof{m_2}$. Then there are no arrows between $m_1$ and $m_2$ i.e.~$\det(m_1,m_2) = 0$. Since $m_1$ and $m_2$ are both primitive, we must have $m_1 = \pm m_2$. But $m_1 \neq m_2$, by the assumption $\sigma_1 \neq \sigma_2$. So $m_1 = -m_2$, which implies $\into{m_1} = \outof{m_2}$. But $\outof{m_2} \neq \into{m_2}$, since polygonal quivers contain no $2$-cycles (Example~\ref{eg:no_self_loops_or_2-cycles}). Therefore $m_1 \not\sim m_2$ in this case as well.

Thus, if $Q := \quiv(N,P)$ then there is a one-to-one correspondence between vertices $v$ of $Q_b$ and maximal cones of $\sigma$ in $\Sigma_P$, as follows:
\begin{equation}\label{eqn:characterization_of_vertices_of_block_quiver}
    v = \{q \in \verts{Q} \mid q\text{ is an inner normal vector to }\sigma\}.
\end{equation}
Since every $\sigma \in \Sigma_P$ has a unique inner normal vector $m$, which is a vertex of $\bquiv(N,P)$, the assignment $\varphi: v \mapsto m$ is a one-to-one correspondence between vertices of $Q_b$ and those of $\bquiv(N,P)$. Furthermore, since every $q \in v$ equals $m$ as an inner normal vector, we have by definition that $\arr{v_1,v_2;Q_b} = \arr{\varphi(v_1),\varphi(v_2);\bquiv(N,P)}$. Therefore, $\varphi$ is arrow-preserving. Finally, if $\varphi(v) = m$, then the labels $(w_v,\ell_v),(w_m,\ell_m)$ of $v,m$ coincide: by definition $w_m$ and $\ell_m$ are the width and local index of a maximal cone $\sigma_m$ in $\Sigma_P$. On the other hand,~\eqref{eqn:characterization_of_vertices_of_block_quiver} characterizes the elements of $v$ as all inner normal vectors to maximal subcones appearing in a standard refinement of $\sigma_m$. All these subcones have local index $\ell_{\sigma}$, so that $\ell_v = \ell_\sigma$ and the sum of their widths is the width of $\sigma$ i.e.~$w_v = w_m$. We conclude that the map $\varphi$ is an isomorphism of decorated quivers.
\end{proof}

\begin{remark}\label{rmk:in_out_covers_all_but_one}
If $Q_b:=\bquiv(N,P)$ is the block quiver of a Fano polygon $P$, and if $v \in \verts{Q_b}$, then the set $\into{v}\cup\outof{v}$ contains all but at most one vertex of $Q_b$ other than $v$. This is immediate from Definition~\ref{def:bquiv_definition} and the observation that, given any edge of $P$, there is at most one other edge that is parallel to it. It follows in particular that a cyclic quiver with at least $5$ vertices can not be the block (or polygonal) quiver of a Fano polygon.
\end{remark}

\subsection{Mutations and the Block Construction}\label{subsec:mutation_group}
For polygonal quivers, passing to the block quiver does not commute with mutation, as the following example shows.

\begin{example}\label{eg:mutation_and_block_do_not_commute}Consider the polygonal quiver $Q$ for $\Proj(1,1,2)$ and its block quiver $Q_b$ as shown in Example~\ref{eg:block_quiver_P1XP1_P112}. Mutating $Q$ at the top-left vertex $m$ and then passing to the block quiver recovers the quiver for $\Proj^1\times\Proj^1$ shown in Figure~\ref{fig:P1xP1_P112_quivers} (top). On the other hand, if we first pass from $Q$ to $Q_b$ then the equivalence class of $m$ is the vertex $v$ of $Q_b$ with label $(2,1)$. Mutating $Q_b$ at $v$ does not recover the quiver of $\Proj^1 \times \Proj^1$.
\end{example}

This discrepancy can be resolved by extending the definition of quiver mutations, as follows:

\begin{definition}\label{def:higher_mutations_of_balanced_quivers}Let $Q$ be a balanced quiver. Choose a vertex $m$ of $Q$ with label $(w,\ell) \in \Z^2$. For any integer $k$, the underlying quiver of $\mut_{m}^{k}(Q)$ is obtained by transforming the underlying quiver of $Q$ in the following manner:
\begin{enumerate}[label=(\alph*),noitemsep]
\item\label{misc:04}
Reverse all arrows incident at $m$.
\item\label{misc:05}
For every subquiver $t \to m \to h$ of $Q$, introduce $k$ arrows $t \to h$.
\item
Following~\ref{misc:04} and~\ref{misc:05}, cancel all $2$-cycles.
\end{enumerate}
A vertex $v \neq m$ of $\mut_{m}^{k}(Q)$ is decorated with the same label as in $Q$ and the vertex $m$ in $\mut_{m}^{k}(Q)$ is decorated with the label $(kD(m) - w, D(m) - \ell)$.
\end{definition}
Notice that setting $k = 1$ recovers Definition~\ref{def:mutation_of_balanced_quivers}: $\mut_{m}^{1}(Q) = \mut_{m}(Q)$. Next, consider the set of pairs $(Q,m)$, with $Q$ a balanced quiver and $m$ a vertex of $Q$. For any integer $k$, define the function $\mut^{k}: (Q,m) \mapsto (\mut_{m}^{k}(Q),m)$.

\begin{prop}\label{prop:mutation_group}If $Q$ is a balanced quiver then $\mut_{m}^{k}(Q)$ is balanced for any $m \in \verts{Q}$ and any integer $k$. Furthermore, $\mut^{0}\circ\mut^{0} = \mathrm{id}$ and
\begin{equation}\label{eqn:group_identity}
  \mut^{t}\circ\mut^{s} = \mut^{0}\circ\mut^{s-t}.
\end{equation}
In particular, the functions $\mut^{k}$ $(k \in \Z)$ are self-inverse. Thus, finite compositions of these functions form a group, which is generated by $\{\mut^{0},\mut^{1}\}$.
\end{prop}

\begin{proof}The statement about balancing follows by repeating the proof of Proposition~\ref{prop:balancing_condition} with minor changes, while the identity~\eqref{eqn:group_identity} follows from a direct calculation using Definition~\ref{def:higher_mutations_of_balanced_quivers} and the observation that $D(m)$ is the same in both $Q$ and $\mut_{m}^{k}(Q)$ for all integers $k$.
\end{proof}

Notice that our group of quiver mutations is not abelian: $\mut^{0}\circ\mut^{1} \neq \mut^{1}\circ\mut^{0}$.

The functions $\mut^{k}$ allow us to mutate block quivers of polygonal quivers in a way that is compatible with mutation: let $Q = \quiv(N,P)$ be a polygonal quiver and let $v$ be a vertex of the block quiver $Q_b$ with label $(w,\ell)$. By Definition~\ref{def:block_quiver_of_decorated_quiver}, $v$ is an equivalence class of vertices of $Q$: if $w = \tau\ell + \rho$, $0 \leq \rho < \ell$, then $v$ contains at most one $R$-vertex $m_0$ and finitely many (possibly zero) primitive $T$-vertices $m_1,\ldots,m_\tau$. The $m_i$ all represent the same inner normal vector $m$ of a maximal cone in $\Sigma_P$. In particular:
\begin{equation}\label{eqn:zero_arrows}
    \arr{m_i,m_j} = 0 \text{ for all }i,j \in \{0,1,\ldots,\tau\}.
\end{equation}
Fix an integer $k$ satisfying $1 \leq k \leq \tau$. One can form a quiver $Q^{k}$ by successively mutating $Q$ (in the sense of Definition~\ref{def:mutation_of_balanced_quivers}) at $k$ of the primitive $T$-vertices: $m_{i_1},\ldots,m_{i_k} \in v - \{m_0\}$. The precise choice of vertices at which to mutate, as well as the order of mutation, is unimportant. The first claim is immediate from Definition~\ref{def:block_quiver_of_decorated_quiver}, while the second follows immediately from~\eqref{eqn:zero_arrows}. After these $k$ mutations, we may pass to the block quiver $Q^{k}_{b}$.

On the other hand, starting from $Q_b$ one may construct an intermediate quiver $Q_{v,k}$. This quiver is identical to $Q_b$ except that the vertex $v$ has been replaced by two vertices: $v_1$ with label $(k\ell,\ell)$ and $v_2$ with label $((\tau-k)\ell + \rho ,\ell)$. All arrows incident at $v$ are removed, $\arr{v_1,v_2;Q_{v,k}}:=0$ and $\arr{x,v_i ; Q_{v,k}} := \arr{x,v;Q_b}$ for $i = 1,2$ and all $x \in \verts{Q_b} - \{v\}$. Passing to the intermediate quiver is the analogue of writing the equivalence class $v$ as $v_1 \cup v_2$ with $v_1 = \{m_{i_1},\ldots,m_{i_k}\}$ and $v_2 = v - v_1$. A direct check now shows that $(\mut_{v_1}^{k}(Q_{v,k}))_{b} = Q^{k}_b$.

\subsection{Hamiltonian Subquivers}\label{subsec:hamiltonian_subquivers}
Given a Fano polygon $P \subset \NQ$, Section~\ref{subsec:degree_derivation} describes a method for constructing a distinguished class of cyclic subquivers of $Q:=\quiv(N,P)$, as explained in Remark~\ref{rmk:cyclic_subquivers}. A similar construction exists for the block quiver $Q_b = \bquiv(N,P)$: label the edges of $P$ cyclically as $E_1,\ldots,E_n$, and let $m_i$ denote the inner normal vector of the cone over $E_i$, so that $\verts{Q_b} = \{m_1,\ldots,m_n\}$.

\begin{definition}\label{def:hamiltonian_subquiver_of_polygonal_block}The \emph{Hamiltonian subquiver} $H$ of $Q_b = \bquiv(N,P)$ has vertex set equal to $\verts{Q_b}$, with $\arr{m_i,m_j;H}:=\arr{m_i,m_j;Q_b}$ if $j \equiv (i+1)\bmod{n}$ and $\arr{m_i,m_j;H}:=0$ otherwise. Every vertex of $H$ carries the same label as in $Q_b$.
\end{definition}

Thus $H$ is a (not necessarily balanced) cyclic subquiver of $Q_b$, with $\verts{H} = \verts{Q_b}$, such that following the vertices of $H$ in the direction of arrows is equivalent to picking out the edges of $P$ in a cyclically ordered sequence. Notice that each $m_i\in\verts{H}$ represents the same inner normal vector as the $m_i\in\verts{C}$ constructed in Section~\ref{subsec:degree_derivation}. It follows that $m_i$ is decorated with the same local index $\ell_i$ in both $H$ and $C$, and $\arr{m_i,m_j;H} = \arr{m_i,m_j;C}$. Thus, repeating the argument of Section~\ref{subsec:degree_derivation} for $Q_b$ with $H$ replacing the cyclic subquiver $C$, we find that $Q = \quiv(N,P)$ and $Q_b = \bquiv(N,P)$ satisfy the (same) quiver degree formula. This is expected, because the underlying polygon $P$ is the same for both $Q$ and $Q_b$.

\begin{lemma}The quiver degree formula~\eqref{eqn:quiver_degree_formula} for a polygonal quiver $Q$ also holds for $Q_b$:
\begin{equation}
\sum_{k = 1}^{n} \frac{\arr{v_k,v_{k+1}}}{\ell_k\ell_{k+1}} = 12 - \tau + \sum_{\sigma \in \basket}{A(\sigma)},
\end{equation}
where $v_1,\ldots,v_n$ are the (cyclically ordered) vertices of the Hamiltonian subquiver of $Q_b$.
\end{lemma}

The quiver degree formula allows us to introduce a class of algebraic varieties that may be interesting from the viewpoint of classification theory. Fix an integer $n \geq 3$ and let $N$ be an oriented lattice. Then:

\begin{prop}[Markov Varieties]\label{prop:markov_varieties}Any Fano $n$-gon $P\subset\NQ$ determines a point on the hypersurface in $\Proj^{2n}\times\mathbb{A}^{1}$ defined by the following polynomial (with indices taken modulo $n$):
\begin{equation}\label{eqn:markov_variety}
    y_1 \ldots y_n\left(z\sum_{i = 1}^{n}\frac{x_i}{y_iy_{i+1}} - t\right).
\end{equation}
Here, $x_1,\ldots,x_n,y_1,\ldots,y_n,z$ are coordinates on $\Proj^{2n}$ and $t$ is the coordinate on $\mathbb{A}^1$.
\end{prop}

\begin{proof}Pass to the polygonal (or block) quiver $Q$ of $P$ and write $\arr{m_{\sigma},m_{\tau};Q} = g(Q)a_{\sigma\tau}$ for any pair of vertices $m_{\sigma},m_{\tau}$ of $Q$. Here, $g(Q)$ ($=g(Q_b)$) is the mutation invariant discussed in Proposition~\ref{prop:arrow_invariant}. The claim now follows from the quiver degree formula, which is the same for both the polygonal and the block quiver of $P$.
\end{proof}

Proposition~\ref{prop:markov_varieties} generalizes results of~\cite{HP10} and~\cite{AK13}, which state that every Fano triangle determines a solution to a Markov-type Diophantine equation.

\section{The Hamiltonian Property}\label{sec:hamiltonian_property}
The Hamiltonian subquiver of a block-polygonal quiver has been defined in Section~\ref{subsec:hamiltonian_subquivers}. A slightly unsatisfactory point is that this definition depends explicitly on the underlying Fano polygon. We will now present a different perspective on Hamiltonian subquivers which, in the polygonal case, allows them to be constructed directly from the block quiver. This construction will play an essential role in Section~\ref{subsec:reconstruction_characterization_general}.

Let $Q_b$ be a balanced block quiver with finitely many vertices, not necessarily coming from a Fano polygon. Fix a vertex $m$ of $Q_b$. For every $x \in \outof{m}$, the \emph{radial distance} from $m$ to $x$ is denoted $\radist{m,x;Q_b} = \radist{m,x}$ and is defined to be the maximal length of a path
\begin{equation}\label{eqn:path_from_m_to_x}
    x_0 = m \rightarrow x_1 \rightarrow \ldots \rightarrow x_k = x
\end{equation}
in $Q_b$ with the constraint that $x_1,\ldots,x_k$ are all distinct elements of $\outof{m}$. The length of~\eqref{eqn:path_from_m_to_x} is $k$ by definition. Note that at least one such path always exists, because $x$ lies in $\outof{m}$, and the length of the longest such path is bounded above, because $\verts{Q_b}$ is finite and $x_1,\ldots,x_k$ are all distinct. Thus, $\radist{m,x}$ is a well-defined positive integer satisfying $1 \leq \radist{m,x} \leq |\outof{m}|$ for all $x\in\outof{m}$. The \emph{radius} of $m$ is defined to be
\[
    r(m):=\min\{r(m,x) \mid x \in \outof{m}\}.
\]
Let $h(m)$ denote the number of vertices $x$ in $\outof{m}$ for which $r(m,x) = r(m)$. In the special case when $r(v) = 1$ and $h(v) = 1$ for all vertices $v$ of $Q_b$, we may define $\seq{m;Q_b} = \seq{m}$ to be the following finite sequence of vertices: $m_1:=m$ and, for $i \geq 1$, $m_{i+1}$ is the unique element of $\outof{m_i}$ for which $\radist{m_i,m_{i+1}} = 1$. The sequence terminates at $m_k$ if $m_{k+1} \in \{m_1,\ldots,m_k\}$.

\begin{definition}\label{def:hamiltonian_property}A balanced block quiver $Q_b$ with finitely many vertices has the \emph{Hamiltonian property} if every vertex $v$ satisfies $r(v) = 1$ and $h(v) = 1$, and there exists a vertex $m$ of $Q_b$ for which $\seq{m}: m_1,\ldots,m_n$ contains all vertices of $Q_b$, with $m_{n+1} = m_1$.
\end{definition}

Note that if $Q_b$ has the Hamiltonian property then, up to cyclic reordering, $\seq{m}$ is independent of the vertex $m$. This implies that the vertices of $Q_b$ can be cyclically ordered, $m_1,\ldots,m_n$, depending on their position in $\seq{m}$.

\begin{definition}\label{def:hamiltonian_subquiver_of_general_block}
Suppose that $Q_b$ has the Hamiltonian property with vertices $m_1,\ldots,m_n$, cyclically labelled as above. The \emph{Hamiltonian subquiver} $H$ of $Q_b$ has vertex set $\{m_1,\ldots,m_n\} = \verts{Q_b}$ with $\arr{m_i,m_j;H} := \arr{m_i,m_j;Q_b}$ if $j = (i+1) \bmod{n}$ and $\arr{m_i,m_j;H}:=0$ otherwise. Ever vertex of $H$ carries the same label as in $Q_b$.
\end{definition}

The Hamiltonian subquiver, if it exists, is unique by construction.

\begin{prop}\label{prop:polygonal_block_has_hamiltonian_property}If $P\subset\NQ$ is a Fano polygon then $Q_b:=\bquiv(N,P)$ has the Hamiltonian property, and the Hamiltonian subquivers from Definitions~\ref{def:hamiltonian_subquiver_of_general_block} and~\ref{def:hamiltonian_subquiver_of_polygonal_block} coincide.
\end{prop}

\begin{proof}Choose a vertex $m$ of $Q_b$. By Definition~\ref{def:bquiv_definition}, $m$ is the inner normal vector of the cone over an edge of $P$. Choose an orientation-preserving isomorphism between $N$ and $\Z^{2}$ such that $m = (0,1)^{t}$. The Fano polygon $P$ is then represented by Figure~\ref{fig:proof_hamiltonian_property}.
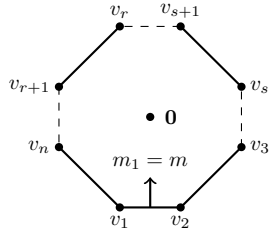
\begin{figure}[h!]
    \centering
    \begin{tikzpicture}[thick,scale=0.8, every node/.style={transform shape}]
        \coordinate [label={right:\small\hspace{1mm}$\orig$}] (O) at (0,0);
        \draw [fill=black] (O) circle (0.05);

        \coordinate [label={below:\small $v_2$}] (v1) at (0.5,-1.5);
        \draw [fill=black] (v1) circle (0.05);
        \coordinate [label={right:\small $v_3$}] (v2) at (1.5,-0.5);
        \draw [fill=black] (v2) circle (0.05);
        \coordinate [label={right:\small $v_s$}] (vs) at (1.5,0.5);
        \draw [fill=black] (vs) circle (0.05);
        \coordinate [label={above:\small $v_{s+1}$}] (vs1) at (0.5,1.5);
        \draw [fill=black] (vs1) circle (0.05);
        \coordinate [label={above:\small $v_r$}] (vr) at (-0.5,1.5);
        \draw [fill=black] (vr) circle (0.05);
        \coordinate [label={left:\small $v_{r+1}$}] (vr1) at (-1.5,0.5);
        \draw [fill=black] (vr1) circle (0.05);
        \coordinate [label={left:\small $v_{n}$}] (vk) at (-1.5,-0.5);
        \draw [fill=black] (vk) circle (0.05);
        \coordinate [label={below:\small $v_1$}] (vk1) at (-0.5,-1.5);
        \draw [fill=black] (vk1) circle (0.05);

        \draw [thick]  (v1) -- (v2);
        \draw [dashed,thin] (v2) -- (vs);
        \draw [thick] (vs) -- (vs1);
        \draw [dashed,thin] (vs1) -- (vr);
        \draw [thick] (vr) -- (vr1);
        \draw [dashed,thin] (vr1) -- (vk);
        \draw [thick] (vk) -- (vk1);
        \draw [thick] (vk1) -- (v1);

        \coordinate (amt) at (0,-1.5);
        \coordinate [label={above:\small $m_1 = m$}] (amh) at (0,-1);
        \draw [->,thick] (amt) -- (amh);
    \end{tikzpicture}
\caption{The Proof of Proposition~\ref{prop:polygonal_block_has_hamiltonian_property}}\label{fig:proof_hamiltonian_property}
\end{figure}
Label the vertices of $P$ cyclically: $v_1,\ldots,v_n$ and let $m_i$ denote the inner normal vector to the cone over the edge $E_i:=\conv{v_i,v_{i+1}}$, $i = 1,\ldots,n$ with indices taken modulo $n$. In this notation, $m = m_1$ and $\outof{m} = \{m_2,\ldots,m_s\}$. The main observation is as follows:
\[
    \arr{m_i,m_j} = \det(m_i,m_j) > 0 \text{ whenever } m_i,m_j \in \outof{m} \text{ and } i < j.
\]
It follows immediately that $r(m,m_k) = k-1$ for $k = 2,\ldots,s$. In particular $r(m) = 1$ and this is attained at precisely one vertex of $Q_b$, namely $m_2$. So $h(m) = 1$. Repeated application of this argument shows that $r(v) = 1$ and $h(v) = 1$ for all vertices $v$ of $Q_b$ and that $\seq{m}$ is $m=m_1,m_2,\ldots,m_n$, with $m_{n+1} = m_1$. Since $\{m_1,\ldots,m_n\} = \verts{Q_b}$, we conclude that $Q_b$ has the Hamiltonian property. Moreover, the ordering on the vertices of $Q_b$ given by $\seq{m}$ agrees with the one coming from cyclically ordering the edges of $P$ in an orientation-preserving manner. Thus, the Hamiltonian subquivers from Definitions~\ref{def:hamiltonian_subquiver_of_polygonal_block} and~\ref{def:hamiltonian_subquiver_of_general_block} coincide.
\end{proof}

Consider a block-polygonal quiver $Q_b = \bquiv(N,P)$. Two key features of the Hamiltonian subquiver $H$ are that $\small\textbf{(1)}$ every oriented cycle of $H$ passes through each vertex of $Q_b$ exactly once (hence the name Hamiltonian) and $\small\textbf{(2)}$ the arrows of $H$ order the vertices of $Q_b$ in a way that coincides with the cyclic ordering of the edges of $P$ induced by the orientation on $N$. It is natural to ask if $\small\textbf{(1)}$ implies $\small\textbf{(2)}$. This is not the case, as shown in Example~\ref{eg:non_hamiltonian_subquiver}.

\begin{example}\label{eg:non_hamiltonian_subquiver}Let $N = \Z^{2}$ with the standard orientation and let $P \subset \Q^2$ be the Fano polygon with vertices $(2,1), (2,3), (1,4), (-1,4), (-2,3), (-2,1), (-1,-1), (1,-1)$. Label the vertices $v_1,\ldots,v_8$ in the (cyclic) order shown and let $m_i$ be the inner normal vector to the edge $\conv{v_i,v_{i+1}}$, with indices taken modulo $8$. The block quiver of $P$ contains the following cyclic subquiver:
\[
    m_1 \rightarrow m_3 \rightarrow m_5 \rightarrow m_7 \rightarrow m_2 \rightarrow m_4 \rightarrow m_6 \rightarrow m_8 \rightarrow m_1
\]
This subquiver is not Hamiltonian. It satisfies condition $\small\textbf{(1)}$ above, but not condition $\small\textbf{(2)}$.
\end{example}

\begin{example}\label{eg:classification_of_small_block_quivers}For a fixed positive integer $n\geq3$, it is often useful to know all types of block quivers that can arise from a Fano $n$-gon. The Hamiltonian property is a useful starting point for such problems. To illustrate this, let $Q_b$ be the block quiver of a Fano triangle ($n = 3$). By Proposition~\ref{prop:polygonal_block_has_hamiltonian_property}, $Q_b$ must have a Hamiltonian subquiver with three vertices as shown in Figure~\ref{fig:block_quiver_triangle}. But we have now drawn all vertices of $Q_b$ and there can be no further arrows between any pair of vertices. So $Q_b$ must equal its Hamiltonian subquiver. In other words, the general block quiver of a Fano triangle can only take the form shown in Figure~\ref{fig:block_quiver_triangle}.

\begin{figure}[h!]
    \centering
    \begin{tikzpicture}[thick,scale=0.8, every node/.style={transform shape}]
      \coordinate (Q1v1) at (0, 0.5);
      \coordinate [label={above:\small $(w_1,\ell_1)$}]  (Q1v1_lab) at (0, 0.65);
      \draw (Q1v1) circle (0.1);
      \draw [fill=black] (Q1v1) circle (0.05);

      \coordinate (Q1v2) at (-1, -1);
      \coordinate [label={below:\small $(w_2,\ell_2)$}]  (Q1v2_lab) at (-1, -1.15);
      \draw [fill=black] (Q1v2) circle (0.05);
      \draw (Q1v2) circle (0.1);

      \coordinate  (Q1v3) at (1, -1);
      \coordinate [label={below:\small $(w_3,\ell_3)$}]  (Q1v3_lab) at (1, -1.15);
      \draw [fill=black] (Q1v3) circle (0.05);
      \draw (Q1v3) circle (0.1);

      \draw [middlearrow={latex}] (Q1v1) -- node[left] {\hspace{-10mm}$\alpha_{12}$} (Q1v2);
      \draw [middlearrow={latex}] (Q1v2) -- node[below] {$\alpha_{23}$} (Q1v3);
      \draw [middlearrow={latex}] (Q1v3) -- node[right] {\hspace{1mm}$\alpha_{31}$} (Q1v1);
\end{tikzpicture}
\caption{The Block Quiver of a Fano Triangle.}\label{fig:block_quiver_triangle}
\end{figure}
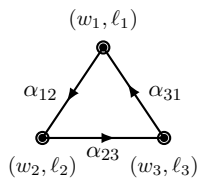

Next, let $Q_b$ be the block quiver of a Fano quadrilateral $P$ ($n=4$). Then $Q_b$ has a Hamiltonian subquiver with four vertices, as shown in Figure~\ref{fig:block_quiver_quadrilateral} (right). It remains to determine whether any arrows can exist between pairs of vertices not joined by arrows in the Hamiltonian subquiver. There are three possibilities, depending on whether $P$ has zero, one or two pairs of parallel edges. Thus, up to relabelling vertices, the block quiver of a Fano quadrilateral can take one of three possible forms, as shown in Figure~\ref{fig:block_quiver_quadrilateral}.

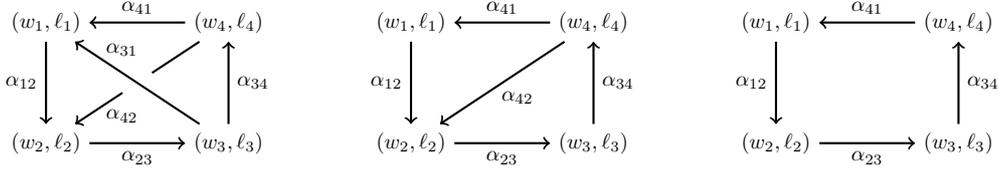
\begin{figure}[h!]
    \centering
    \captionsetup{justification=centering}
    \begin{tikzpicture}[thick,scale=0.8, every node/.style={transform shape}]
      \node at (-7.5,1)   (Q0v1) {\small$(w_1,\ell_1)$};
      \node at (-7.5,-1)  (Q0v2) {\small$(w_2,\ell_2)$};
      \node at (-4.5,-1)  (Q0v3) {\small$(w_3,\ell_3)$};
      \node at (-4.5,1)   (Q0v4) {\small$(w_4,\ell_4)$};
      \node at (-5.90,0.07) (blank1) {};
      \node at (-6.10,-0.07) (blank2) {};
      \node at (-6.25,-0.55) (lab42Q0) {\small$\alpha_{42}$};
      \node at (-6.25,0.55)  (lab31Q0) {\small$\alpha_{31}$};
      \draw [->, thick] (Q0v1) -- node[left]  {\small$\alpha_{12}$} (Q0v2);
      \draw [->, thick] (Q0v2) -- node[below] {\small$\alpha_{23}$} (Q0v3);
      \draw [->, thick] (Q0v3) -- node[right] {\small$\alpha_{34}$} (Q0v4);
      \draw [->, thick] (Q0v4) -- node[above] {\small$\alpha_{41}$} (Q0v1);
      \draw [->, thick] (Q0v3) -- (Q0v1);
      \draw [-, thick]  (Q0v4) -- (blank1);
      \draw [->, thick] (blank2) -- (Q0v2);

      \node at (-1.5,1)  (Q1v1) {\small$(w_1,\ell_1)$};
      \node at (-1.5,-1) (Q1v2) {\small$(w_2,\ell_2)$};
      \node at (1.5,-1)  (Q1v3) {\small$(w_3,\ell_3)$};
      \node at (1.5,1)   (Q1v4) {\small$(w_4,\ell_4)$};
      \node at (0.25,-0.25)   (lab42Q1) {\small$\alpha_{42}$};
      \draw [->, thick] (Q1v1) -- node[left]  {\small$\alpha_{12}$} (Q1v2);
      \draw [->, thick] (Q1v2) -- node[below] {\small$\alpha_{23}$} (Q1v3);
      \draw [->, thick] (Q1v3) -- node[right] {\small$\alpha_{34}$} (Q1v4);
      \draw [->, thick] (Q1v4) -- node[above] {\small$\alpha_{41}$} (Q1v1);
      \draw [->, thick] (Q1v4) -- (Q1v2);

      \node at (4.5,1)  (Q2v1) {\small$(w_1,\ell_1)$};
      \node at (4.5,-1) (Q2v2) {\small$(w_2,\ell_2)$};
      \node at (7.5,-1) (Q2v3) {\small$(w_3,\ell_3)$};
      \node at (7.5,1)  (Q2v4) {\small$(w_4,\ell_4)$};
      \draw [->, thick] (Q2v1) -- node[left]  {\small$\alpha_{12}$} (Q2v2);
      \draw [->, thick] (Q2v2) -- node[below] {\small$\alpha_{23}$} (Q2v3);
      \draw [->, thick] (Q2v3) -- node[right] {\small$\alpha_{34}$} (Q2v4);
      \draw [->, thick] (Q2v4) -- node[above] {\small$\alpha_{41}$} (Q2v1);
    \end{tikzpicture}
\caption{The Block Quiver of a Fano Quadrilateral with zero (left), one (middle) and two (right) pairs of parallel edges.}\label{fig:block_quiver_quadrilateral}
\end{figure}
\end{example}

\begin{example}[A Non-Existence Result]\label{eg:non_existence_of_fanos_with_prescribed_singularities}Let $P \subset \NQ$ be a Fano triangle with singularity content $(\tau,\basket)$. By Example~\ref{eg:classification_of_small_block_quivers}, we know that the block quiver $Q_b$ of $P$ takes the general form shown in Figure~\ref{fig:block_quiver_triangle}. In the notation of Figure~\ref{fig:block_quiver_triangle}, assume that $P$ has \emph{coprime widths}: $\gcd{w_i,w_j} = 1$ whenever $i \neq j$. The balancing condition at each vertex of $Q_b$ gives the following system of linear equations:
\[
w_1\alpha_{12} - w_3\alpha_{23} = 0 \quad;\quad w_2\alpha_{23} - w_1\alpha_{31} = 0 \quad;\quad w_3\alpha_{31} - w_2\alpha_{12} = 0.
\]
Since the $w_i$ are pairwise coprime, the space of integer solutions $(\alpha_{12},\alpha_{23},\alpha_{31})$ to this system is $\Z$-spanned by $(w_3,w_1,w_2)$, so that
\[
    (\alpha_{12},\alpha_{23},\alpha_{31}) = g\cdot(w_3,w_1,w_2),
\]
where $g = g(Q_b)$ is the mutation invariant discussed in Proposition~\ref{prop:arrow_invariant}. Thus, in the case of coprime widths, the quiver degree formula for $Q_b$ specializes to:
\begin{equation}\label{eqn:quiver_degree_formula_triangles}
  g(w_1\ell_1 + w_2\ell_2 + w_3\ell_3) = \left(12 - \tau + \sum_{\sigma \in \basket}{A(\sigma)}\right)\ell_1\ell_2\ell_3.
\end{equation}
As an application, there can not exist a Fano triangle $P$, such that the maximal cones of $\Sigma_P$ are smooth ($w_1 = 1 = \ell_1$), a $\frac{1}{3}(1,1)$ singularity ($w_2 = 1, \ell_2 = 3$) and a $\frac{1}{4}(1,1)$ singularity ($w_2 = 2 = \ell_2$). Such a $P$ would have singularity content $(2,\{\frac{1}{3}(1,1)\})$, and since $A(\frac{1}{3}(1,1)) = -\frac{5}{3}$, it would follow from Equation~\eqref{eqn:quiver_degree_formula_triangles} that there is an integer $g$ satisfying
\[
    g(1\cdot1 + 1\cdot3 + 2\cdot2) = (12 - 2 - (5/3))1\cdot3\cdot2 \quad \text{i.e.}\quad 8g = 50.
\]
This is a contradiction, so no such $P$ exists. Geometrically, this means that there does not exist a toric Fano surface of Picard rank $1$ with isolated singularities $\frac{1}{3}(1,1)$ and $\frac{1}{4}(1,1)$.
\end{example}

\section{Characterization of Polygonal Quivers}\label{sec:reconstruction_and_characterization}
The present discussion has been about quivers arising from Fano polygons. We have seen that polygonal quivers are balanced (Proposition~\ref{def:mutation_of_balanced_quivers}), but not every balanced quiver is polygonal (Example~\ref{eg:family_of_non_polygonal_quivers}). It is natural to ask which balanced quivers arise from Fano polygons. We will first demonstrate an ad hoc approach to this question (Example~\ref{eg:reconstruction_for_P113}). This is more elementary than the general discussion starting in Section~\ref{subsec:reconstruction_characterization_triangles}, and may be useful in the study of individual examples. See also Example~\ref{eg:reconstruction_for_P113_using_result}.

\begin{example}\label{eg:reconstruction_for_P113}
We will investigate whether the following balanced quiver $Q$ is polygonal.

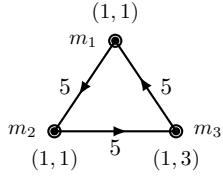
\begin{figure}[h!]
    \centering
    \begin{tikzpicture}[thick,scale=0.8, every node/.style={transform shape}]
      \coordinate (Q1v1) at (0, 0.5);
      \coordinate [label={above:\small $(1,1)$}]  (Q1v1_lab) at (0, 0.65);
      \coordinate [label={left:\small $m_1$}]  (Q1v1_lab_m) at (-0.15, 0.5);
      \draw (Q1v1) circle (0.1);
      \draw [fill=black] (Q1v1) circle (0.05);

      \coordinate (Q1v2) at (-1, -1);
      \coordinate [label={below:\small $(1,1)$}]  (Q1v2_lab) at (-1, -1.15);
      \coordinate [label={left:\small $m_2$}]  (Q1v2_lab_m) at (-1.15, -1);
      \draw [fill=black] (Q1v2) circle (0.05);
      \draw (Q1v2) circle (0.1);

      \coordinate  (Q1v3) at (1, -1);
      \coordinate [label={below:\small $(1,3)$}]  (Q1v3_lab) at (1, -1.15);
      \coordinate [label={right:\small $m_3$}]  (Q1v3_lab_m) at (1.15, -1);
      \draw [fill=black] (Q1v3) circle (0.05);
      \draw (Q1v3) circle (0.1);

      \draw [middlearrow={latex}] (Q1v1) -- node[left] {\hspace{-6mm}$5$} (Q1v2);
      \draw [middlearrow={latex}] (Q1v2) -- node[below] {$5$} (Q1v3);
      \draw [middlearrow={latex}] (Q1v3) -- node[right] {\hspace{1mm}$5$} (Q1v1);
\end{tikzpicture}
\end{figure}

As a first step, pass to the block quiver $Q_b$ to simplify calculations. Note that $Q \cong Q_b$ in this example, but that this step would be nontrivial if one were to apply this calculation, for instance, to the quiver shown in Figure~\ref{fig:P1xP1_P112_quivers} (bottom). The problem to consider now is whether $Q_b$ is the block quiver (Definition~\ref{def:bquiv_definition}) of some Fano polygon $P$.

Suppose there exists a Fano polygon $P \subset \NQ$ for which $Q_b$ is a block quiver. Then the vertices $m_1,m_2,m_3$ of $Q$, as shown above, represent primitive vectors in the dual lattice $M$. Fix an isomorphism of $M$ with $\Z^2$ such that
\[
    m_1 = (0,1)^{t}, \quad m_2 = (x_2,y_2)^{t}, \quad m_3 = (x_3,y_3)^{t},
\]
where the $x_i,y_j$ are at present unknown. The (signed) number of arrows between $m_i$ and $m_j$ is given by the equation $\arr{m_i,m_j ; Q_b} = \det(m_i,m_j)$. Applying this to each vertex of $Q$ gives a system of three equations, from which it follows that
\[
    m_1 = (0,1)^{t}, \quad m_2 = (-5,y_2)^{t}, \quad m_3 = (5, -1 - y_2)^{t}.
\]
The quantity $y_2$ is still unknown. Now apply the change of basis $\varphi_k:M \to M, (x,y)^{t} \mapsto (x, y - kx)^{t}$, for some integer $k$ to be fixed later. In the new basis:
\[
    m_1 = (0,1)^{t}, \quad m_2 = (-5, y_2 + 5k)^{t}, \quad m_3 = (5, -1 - (y_2 + 5k))^{t}.
\]
The primitivity of $m_2$ and $m_3$ implies that $y_2$ is not congruent to $0$ or $4$ modulo $5$. So there are three cases to consider. First, assume that $y_2 \equiv 1 \bmod{5}$. Then, after fixing a suitable integer $k$, we arrive at a basis in which
\[
    m_1 = (0,1)^{t}, \quad m_2 = (-5, 1)^{t}, \quad m_3 = (5, -2)^{t}.
\]
Since the $m_i$ are being interpreted as inner normal vectors, consider the three hyperplanes in $\NQ$ defined by the equations: $\langle m_i, z \rangle = -\ell_i$, where $z \in \NQ$, $\langle\cdot,\cdot\rangle$ is the standard pairing and $(w_i,\ell_i)$ is the label of $m_i$ in $Q_b$ for $i = 1,2,3$. The vertices of $P$ are solutions $(u,v) \in \Z^2$ to pairs of these equations. Explicitly, the equations in this basis are
\begin{equation}\label{eqn:hyperplane_equations}
    v = -1 \quad;\quad  -5u + v = -1 \quad;\quad 5u-2v = -3.
\end{equation}
Solving these pairwise, we obtain the vertex set $\{(0,-1),(-1,-1),(1,4)\}$ for $P \subset \Q^2$. Since there is a linear relation: $1\cdot(1,4) + 1\cdot(-1,-1) + 3\cdot(0,-1) = \orig$, and since $(0,-1)$ and $(-1,-1)$ span the ambient lattice $\Z^{2}$, we conclude that the associated toric variety is weighted projective space $\Proj(1,1,3)$. A direct check shows that the polygonal quiver of $P$ is isomorphic to $Q$. The two remaining cases: $y_2 \equiv 2,3\bmod{5}$ do not yield any Fano polygons, because the hyperplane equations analogous to~\eqref{eqn:hyperplane_equations} in these cases contain pairs which can not be solved simultaneously over $\Z^2$. We conclude that the quiver $Q$ is polygonal, and it determines a unique Fano polygon, up to the action of $SL_2(\Z)$.
\end{example}

\subsection{The Triangle Case}\label{subsec:reconstruction_characterization_triangles}
Let $Q$ be a balanced quiver whose block quiver $Q_b$ has three vertices. Suppose~$\lab{1}$ $Q_b$ is cyclic, as shown in Figure~\ref{fig:block_quiver_triangle}, with $\alpha_{ij}\in \Z_{\geq 1}$. Let $m_i$ denote the vertex with label $(w_i,\ell_i)$, so that $\arr{m_i,m_j;Q_b} = \alpha_{ij}$. Suppose~$\lab{2}$ $w_i,\ell_i \geq 1$ for $i = 1,2,3$. Choose a vertex of $Q_b$: up to cyclic relabelling, let this vertex be $m_1$. Define integers:
\begin{equation}\label{eqn:y_coordinates_triangles}
    y_1 := -\ell_1 \quad;\quad y_2 := -\ell_1 \quad;\quad y_3 := y_2 + w_2\cdot\arr{m_1,m_2}.
\end{equation}
Assume $\lab{3}$~$y_3 > 0$ and~$\lab{4}$ there exists an $x \in \Z$ satisfying $\gcd{x,y_1} = \gcd{x+w_1,y_2} = 1$. For such an $x$, define the following rational numbers, with indices taken modulo $3$:
\begin{align*}
    x_1 := x \quad;\quad &x_2 := x + w_1 \quad;\quad x_3:=(w_3\ell_3 + x_1y_3)/y_1 \quad;\quad x_3' := (x_2y_3 - w_2\ell_2)/y_2  \\ &s_j := -\arr{m_1,m_j} \quad;\quad t_j:=(x_{j+1} - x_j)/w_j \quad \text{for } j = 2,3.
\end{align*}
Assume~$\lab{5}$ $t_2,t_3 \in \Z$,~$\lab{6}$ $x_3 = x_3'$ and~$\lab{7}$ $\gcd{x_3,y_3} = \gcd{s_j,t_j} = 1$ for $j = 2,3$.

\begin{prop}\label{prop:reconstruction_triangles}If $Q_b$ satisfies conditions~$\lab{1}$ to~$\lab{7}$ for some vertex $m_1$ then it is the block quiver of a Fano triangle. Conversely, if $Q_b$ is the block quiver of a Fano triangle then there is a vertex $m_1$ for which conditions~$\lab{1}$ to~$\lab{7}$ are satisfied.
\end{prop}

\begin{proof}Unless otherwise stated, all indices will be taken modulo $3$. Suppose that $Q_b$ satisfies conditions~$\lab{1}$ to~$\lab{7}$ for some vertex $m_1$. Let $N = \Z^{2}$ with the standard orientation. Let $v_i := (x_i,y_i) \in \NQ$ for $i = 1,2,3$ and let $P:=\conv{v_1,v_2,v_3} \subset \NQ$. We will show that $P$ is a Fano triangle and $Q_b \cong \bquiv(N,P)$.

Condition~$\lab{5}$ implies that $x_3$ is an integer. Thus, by~$\lab{4}$ and~$\lab{7}$, the $v_i$ are primitive lattice vectors in $N$. These vectors are all distinct: $v_1 \neq v_2$ since $x_2 - x_1 = w_1 \geq 1$ by~$\lab{2}$. Similarly $v_1,v_2$ are both distinct from $v_3$ because their $y$-coordinates have different signs: $y_1 = y_2 = -\ell_1$, which is negative by~$\lab{2}$, while $y_3$ is positive by~$\lab{3}$. This observation also shows that $v_1,v_2$ both lie on the line $\{(x,y) \mid y = -\ell_1\} \subset \NQ$, but this line does not contain $v_3$. So $P$ is a $2$-dimensional lattice polytope in $\NQ$ whose vertices are primitive lattice vectors. Now let $M:=\Hom{N,\Z}$ and let $m^{1}:=(0,1)^{t}$, $m^{2}:=(s_2,t_2)^{t}$, $m^{3}:=(s_3,t_3)^{t} \in \MQ$. By~$\lab{5}$ and~$\lab{7}$, the $m^{i}$ are primitive lattice vectors in $M$. A direct calculation shows that:
\begin{equation}\label{eqn:inner_normal_equations_triangles}
        \langle m^{i},v_i \rangle = \langle m^{i},v_{i+1} \rangle = -\ell_i \text{ for } i=1,2,3,
\end{equation}
where $\langle \cdot,\cdot \rangle: M \times N \to \Z$ is the natural pairing. Condition~$\lab{6}$ is used to show~\eqref{eqn:inner_normal_equations_triangles} for $i = 2$. Similarly the shape of $Q_b$, assumed in~$\lab{1}$, is used to show $w_2\cdot\arr{m_1,m_2} = w_3\cdot\arr{m_3,m_1}$ by balancing at $m_1$. This is used to show~\eqref{eqn:inner_normal_equations_triangles} for $i = 3$. By condition~$\lab{2}$, the equations~\eqref{eqn:inner_normal_equations_triangles} show that $m^{i}$ must be the inner normal vector of (the cone over) the edge $\conv{v_i,v_{i+1}}$ of $P$. In particular, the origin must lie in the strict interior of $P$, showing that $P$ is a \emph{Fano} triangle.

By Definition~\ref{def:bquiv_definition}, $\bquiv(N,P)$ has three vertices $m^{1},m^{2}$ and $m^{3}$. In the present basis, the primitive vectors in directions $v_2-v_1$, $v_3-v_2$ and $v_1-v_3$ are $(1,0), (t_2,-s_2)$ and $(t_3,-s_3)$ respectively. From this we see that the cone over the edge $\conv{v_i,v_{i+1}}$ of $P$ has width $w_i$. Combining this with~\eqref{eqn:inner_normal_equations_triangles} shows that the label of $m^{i}$ in $\bquiv(N,P)$ is $(w_i,\ell_i)$. This equals the label of $m_i$ in $Q_b$. A direct calculation of determinants also shows that:
\[
    \arr{m^{i},m^{j}; \bquiv(N,P)} = \arr{m_i,m_j ; Q_b} \text{ for } i,j \in \{1,2,3\}.
\]
Therefore we conclude that $Q_b \cong \bquiv(N,P)$ i.e.~$Q_b$ is the block quiver of a Fano triangle.

Conversely, suppose $Q_b \cong \bquiv(N,P)$ for some some Fano polygon $P \subset \NQ$. Since $Q_b$ has three vertices, $P$ must be a triangle. By fixing a suitable isomorphism, we may assume that $N = \Z^{2}$ with the standard orientation. Label the vertices of $P$ cyclically: $v_1,v_2,v_3$, respecting the orientation i.e.~ so that $\det(v_i,v_{i+1})>0$. Let $m^{i} \in M$ denote the inner normal vector of the cone over the edge $\conv{v_i,v_{i+1}}$. Denote the width and local index of this cone by $w_i,\ell_i$ respectively. We may assume that $m^{1} = (0,1)^{t}$ in the current basis. Write $v_i = (x_i,y_i)$ for $i = 1,2,3$. We will show that conditions~$\lab{1}$ to~$\lab{7}$ are satisfied.

By Definition~\ref{def:bquiv_definition}, the vertices of $Q_b$ are $m^{1},m^{2}$ and $m^{3}$, and the label of each $m^{i}$ is $(w_i,\ell_i)$. The width and local index of a cone are always positive integers. Thus condition~$\lab{2}$ holds. Moreover, arguing as in our proof of the balancing condition (Proposition~\ref{prop:balancing_condition}), we see that:
\begin{equation}\label{eqn:arrows_height_interpretation}
    w_{i+1}\cdot\arr{m^{i},m^{i+1}} = \langle m^{i},v_{i+2} \rangle - \langle m^{i},v_{i+1} \rangle.
\end{equation}
Since $m^{i}$ is the inner normal to the cone over $\conv{v_i,v_{i+1}}$, the quantity $\langle m^{i},v_{i+1} \rangle$ must be negative. On the other hand, since $P$ contains the origin in its strict interior, $\langle m^{i},v_{i+2} \rangle$ must be positive. Thus $\arr{m^{i},m^{i+1}}$ is positive for $i = 1,2,3$. We conclude that $Q_b$ is cyclic and condition~$\lab{1}$ holds. Observing that $m^{1} = (0,1)$ and setting $i = 1$ in~\eqref{eqn:arrows_height_interpretation} we see that the $y$-coordinates of $v_i,v_2,v_3$ are described by~\eqref{eqn:y_coordinates_triangles}. In particular $y_3 > 0$, so~$\lab{3}$ holds. Moreover, the edge/line segment joining $v_1 = (x_1,y_1) = (x_1,-\ell_1)$ and $v_2 = (x_2,y_2) = (x_2,-\ell_2)$ has lattice length $w_i$, so $x_2-x_1 = w_1$. Since $v_1$ and $v_2$ are primitive lattice vectors, the integer $x:=x_2$ satisfies condition~$\lab{4}$. Next, arguing as in Example~\ref{eg:normalized_volume_from_quiver}, we see that:
\begin{equation}\label{eqn:volume_equation_reconstruction_triangles}
    w_i\ell_i = \det(v_i,v_{i+1}) = x_iy_{i+1} - x_{i+1}y_i \text{ for } i = 1,2,3.
\end{equation}
Equation~\eqref{eqn:volume_equation_reconstruction_triangles} for $i = 2$ and $i=3$ can be used to obtain two different expressions for $x_3$. Since both of these must be equal, condition~$\lab{6}$ is satisfied. Finally, a direct calculation of the inner normals of $P$ shows that $m^{2} = (s_2,t_2)$ and $m_3 = (s_3,t_3)$, where the $s_j,t_j$ are as defined above. This immediately shows condition~$\lab{5}$, that $t_2,t_3$ are integers. Condition~$\lab{7}$ now follows from the primitivity of $v_3$ and of $m^2,m^3$. Thus, there exists a vertex of $Q_b$ for which conditions~$\lab{1}$ to~$\lab{7}$ are satisfied.
\end{proof}

\begin{remark}In the above discussion, note that the conditions~$\lab{1}$ to~$\lab{7}$ and quantities $x_i,y_i$ etc.~are all expressed in terms of a nominated vertex of $Q_b$. Suppose that $Q_b$ satisfies~$\lab{1}$ to~$\lab{7}$ with respect to the vertex $m_1$. Then we may construct a Fano triangle $P$ such that the vertices $m_1,m_2,m_3$ of $Q_b$ are identified with the inner normal vectors $m^{1},m^{2},m^{3}$ of the maximal cones in $\Sigma_P$. By choosing a basis such that either $m^{2}$ or $m^{3}$ equals $(0,1)^{t}$ and arguing as in the converse part of Proposition~\ref{prop:reconstruction_triangles}, we see that $Q_b$ also satisfies ~$\lab{1}$ to~$\lab{7}$ with respect to $m_2$ and $m_3$, once the quantities $x_i,y_i$ etc.~have been redefined in terms of the new vertices. Thus, the choice of nominated vertex does not affect whether or not $Q_b$ is block-polygonal.
\end{remark}

\begin{example}\label{eg:reconstruction_for_P113_using_result}Consider the (block) quiver $Q_b$ of Example~\ref{eg:reconstruction_for_P113}, with vertices $m_1,m_2,m_3$ as shown. A direct calculation shows that $Q_b$ satisfies conditions~$\lab{1}$ to~$\lab{7}$ with respect to the vertex $m_1$ and $x:=0$. Therefore, by (the proof of) Proposition~\ref{prop:reconstruction_triangles}, $Q_b$ is the block quiver of the Fano triangle in $\Z^{2}$ with vertices $v_1:=(0,-1), v_2:=(1,-1)$ and $v_3:=(-3,4)$. This triangle is $GL_2(\Z)$-equivalent to the one obtained in Example~\ref{eg:reconstruction_for_P113}.
\end{example}

\begin{remark}[Expected Volume]\label{rmk:expected_volume_triangles}Recall condition~$\lab{6}$ from Proposition~\ref{prop:reconstruction_triangles}, which requires that $x_3 = x_3'$. Using the definitions of $x_3,x_3',x_i,y_j$ etc. we see that this condition, with respect to the vertex $m_1$, is equivalent to the following equality:
\begin{equation}\label{eqn:expected_volume_triangles}
    w_1\ell_1 + w_2\ell_2 + w_3\ell_3 = w_1w_2\cdot \arr{m_1,m_2}.
\end{equation}
Note that this expression is independent of the integer $x$ chosen in condition~$\lab{4}$. The right hand side of~\eqref{eqn:expected_volume_triangles} appears to depend on the vertex $m_1$: indeed, writing condition~$\lab{6}$ for $m_2$ would replace this term by $w_2w_3\cdot\arr{m_2,m_3}$, and for $m_3$ we would have $w_1w_3\cdot\arr{m_3,m_1}$. But these three right hand side terms are equal in the triangle case, by the balancing condition (Proposition~\ref{prop:balancing_condition}). Thus,~\eqref{eqn:expected_volume_triangles} is a \emph{global} condition on a three-vertex cyclic balanced quiver.

If we assume that the block quiver of Proposition~\ref{prop:reconstruction_triangles} comes from a Fano triangle $P$, the equality~\eqref{eqn:expected_volume_triangles} can be given a natural interpretation: a slight modification of Example~\ref{eg:normalized_volume_from_quiver} identifies the left hand side of~\eqref{eqn:expected_volume_triangles} as the normalized volume of $P$, as computed from the block quiver. On the other hand the right hand side, $w_1w_2\cdot\arr{m_1,m_2}$ can be rewritten as $w_1\cdot D(m_1)$. By thinking of $w_1$ as the `base length' of $P$ and $D(m_1)$ as the `height' of $P$, this is the well-known formula for the normalized volume of a triangle. Condition~$\lab{6}$, rewritten as~\eqref{eqn:expected_volume_triangles}, is then just the statement that these two calculations of normalized volume agree i.e.~the \emph{expected volume} of a Fano triangle underlying our block quiver is well-defined.
\end{remark}

\begin{example}As an application of Remark~\ref{rmk:expected_volume_triangles}, consider the (block) quiver $Q_b$ shown in Figure~\ref{fig:polygonal_quivers_are_not_closed_under_mutation} (right). For this quiver, the left hand side of~\eqref{eqn:expected_volume_triangles} equals: $2\cdot1 + 1\cdot1 + 1\cdot1 = 4$, while the right hand side, computed at the vertex $(2,1)$, equals: $2\cdot1\cdot4 = 8$. We conclude that $Q_b$ is not the block quiver of a Fano polygon, because the expected volume is not well-defined.
\end{example}

\begin{example}[Families of Non-Polygonal Quivers]\label{eg:family_of_non_polygonal_quivers} The balanced quiver $Q_1$ shown in Figure~\ref{fig:families_of_non_polygonal_quivers} (left) is polygonal if and only if $a=3$: if $a=3$, then $Q_1$ is obtained from the Fano polygon of $\Proj^2$ in $\Q^2$ with vertex set $\{(1,0),(0,1),(-1,-1)\}$.

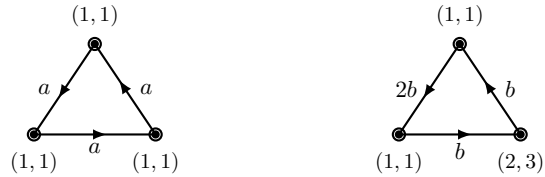
\begin{figure}[h!]
    \centering
    \begin{tikzpicture}[thick,scale=0.8, every node/.style={transform shape}]
      \coordinate  (P2v1) at (-2, -1);
      \coordinate [label={below:\small $(1,1)$}]  (P2v1_lab) at (-2, -1.15);
      \coordinate  (P2v2) at (-4, -1);
      \coordinate [label={below:\small $(1,1)$}]  (P2v2_lab) at (-4, -1.15);
      \coordinate  (P2v3) at (-3, 0.5);
      \coordinate [label={above:\small $(1,1)$}]  (P2v3_lab) at (-3, 0.65);

      \coordinate  (P116v1) at (4, -1);
      \coordinate [label={below:\small $(2,3)$}]  (P116v1_lab) at (4, -1.15);
      \coordinate  (P116v2) at (2, -1);
      \coordinate [label={below:\small $(1,1)$}]  (P116v2_lab) at (2, -1.15);
      \coordinate  (P116v3) at (3, 0.5);
      \coordinate [label={above:\small $(1,1)$}]  (P116v3_lab) at (3, 0.65);

      \draw [fill=black] (P2v1) circle (0.05);
      \draw (P2v1) circle (0.1);
      \draw [fill=black] (P2v2) circle (0.05);
      \draw (P2v2) circle (0.1);
      \draw [fill=black] (P2v3) circle (0.05);
      \draw (P2v3) circle (0.1);

      \draw [fill=black] (P116v1) circle (0.05);
      \draw (P116v1) circle (0.1);
      \draw [fill=black] (P116v2) circle (0.05);
      \draw (P116v2) circle (0.1);
      \draw [fill=black] (P116v3) circle (0.05);
      \draw (P116v3) circle (0.1);

      \draw [middlearrow={latex}] (P2v2) -- node[below] {$a$} (P2v1);
      \draw [middlearrow={latex}] (P2v1) -- node[right] {\hspace{1mm}$a$} (P2v3);
      \draw [middlearrow={latex}] (P2v3) -- node[left] {\hspace{-6mm}$a$} (P2v2);

      \draw [middlearrow={latex}] (P116v2) -- node[below] {$b$} (P116v1);
      \draw [middlearrow={latex}] (P116v1) -- node[right] {\hspace{1mm}$b$} (P116v3);
      \draw [middlearrow={latex}] (P116v3) -- node[left] {\hspace{-7mm}$2b$} (P116v2);

   \end{tikzpicture}
   \caption{Two Families of Non-Polygonal Quivers ($a\neq3$ and $b\neq4$).}\label{fig:families_of_non_polygonal_quivers}
\end{figure}
Conversely, if $Q_1$ is polygonal then (since it equals its own block quiver) it must satisfy Equation~\eqref{eqn:expected_volume_triangles}, from which it follows that $9 = 3a$. A similar argument shows that the balanced quiver $Q_2$ shown in Figure~\ref{fig:families_of_non_polygonal_quivers} (right) is polygonal if and only if $b=4$. In this case, $Q_2$ is obtained from the Fano polygon of $\Proj(1,1,6)$ in $\Q^2$ with vertex set $\{(1,0),(0,1),(-1,-6)\}$.
\end{example}

\subsection{The General Case}\label{subsec:reconstruction_characterization_general}
Consider the block quiver $Q_b$ of a polygonal quiver $Q$. Suppose that~$\lab{1}$ for every vertex $v$ of $Q_b$, the set $\into{v}\cup\outof{v}$ contains all but at most one vertex of $Q_b$ other than $v$ and~$\lab{2}$ $Q_b$ has the Hamiltonian property. Choose a vertex $m_1$ of $Q_b$. Label the vertices in $\outof{m_1}$ as $m_2,\ldots,m_{k-1}$ and those in $\into{m_1}$ as $m_{k+1},\ldots,m_n$. The ordering on the $m_i$ as $i$ increases should coincide with that given by $\seq{m_1}$. If there is a vertex of $Q_b$ different from $m_1$ which does not lie in $\into{m_1}\cup\outof{m_1}$, label it $m_k$. Otherwise, $m_k$ will remain undefined. Let $(w_i,\ell_i) \in \Z^2$ denote the label of $m_i$ for $i = 1,\ldots,n$ and assume that~$\lab{3}$ $w_1,\ell_i \geq 1$ for $i = 1,\ldots,n$. Taking indices modulo $n$, define integers:
\begin{align*}
&y_1 := -\ell_1 &;\quad& y_2 := -\ell_1 &;&\quad y_{i+1} := y_i + w_1\cdot\arr{m_1,m_i} \text{ for } i = 2,\ldots,k-1;\\
&s_1 := 0       &;\quad& t_1 := 1       &;&\quad y_j:= y_{j+1} + w_j\cdot\arr{m_j,m_1} \text{ for } j = n,\ldots,k+1.
\end{align*}
Let $s_k:=0$ and $t_k:=-1$ if $m_k$ exists. Note that $y_k = y_{k+1}$ by the balancing condition at $m_1$. Assume that~$\lab{4}$ $y_k = \ell_k$ if $m_k$ exists and $y_k > 0$ otherwise. Then $y_1,y_2,y_k$ and $y_{k+1}$ are all nonzero. Observe that $y_2,\ldots,y_k$ is strictly increasing and $y_{k+1},\ldots,y_n$ is strictly decreasing.

\begin{remark}There are now four possibilities, depending on whether or not one of $y_3,\ldots,y_{k-1}$ is zero and whether or not one of $y_{k+2},\ldots,y_n$ is zero. In what follows, we will consider the case when one of $y_3,\ldots,y_{k-1}$ is zero and all of $y_{k+2},\ldots,y_n$ are nonzero. The remaining cases can be treated similarly.
\end{remark}

Consider the case when $y_r = 0$ for some fixed $r$ satisfying $3 \leq r \leq k-1$, and all other $y_i$ are nonzero. Assume that~$\lab{5}$ there exists an integer $x$ satisfying $\gcd{x,y_1} = 1$ and $\gcd{x+w,y_2} = 1$. Let $x_{r+1}$ be an arbitrary integer (to be defined later). Define the following rational numbers:
\begin{align*}
&x_1:=x &;\quad& x_2:=x + w_1 &;\quad& x_{i+1}:= (x_iy_{i+1} - w_i\ell_i)/y_i \text{ for }i = 2,\ldots,\widehat{r},\ldots,k-1;\\
&x'_r:=\frac{w_r\ell_r}{y_{r+1}} &;\quad& s_i:= -\arr{m_1,m_i} &;\quad& t_i:= (x_{i+1}-x_i)/w_i \text{ for } i = 2,\ldots,k-1.
\end{align*}
Define the following rational numbers (independent of $x_{r+1}$), with indices taken modulo $n$:
\begin{equation*}
x_j := (x_{j+1}y_j + w_j)/y_{j+1} \quad;\quad s_j := \arr{m_j,m_1} \quad;\quad t_j := (x_{j+1}-x_j)/w_j \text{ for }j=n,\ldots,k+1.
\end{equation*}
Assume~$\lab{6}$ $x_r = x'_r$ and that~$\lab{7}$ there exists an integer $x_{r+1}$ such that the following conditions are satisfied:~$\lab{7a}$ $t_i$ is an integer for $i = 2,\ldots,n$,~$\lab{7b}$ $\gcd{x_i,y_i} = 1$ and $\gcd{s_i,t_i} = 1$ for $i = 2,\ldots,n$ and~$\lab{7c}$ $x_k = x_{k+1} + w_k$ if $m_k$ exists and $x_k = x_{k+1}$ otherwise. Assume that~$\lab{8}$ $s_it_j - s_jt_i = \arr{m_i,m_j}$ whenever $j \neq i-1,i+1$ and~$\lab{9}$ $s_ix_j + t_iy_j > -\ell_i$ whenever $j \neq i,i+1$.

\begin{thm}\label{thm:reconstruction_general}If $Q_b$ satisfies conditions~$\lab{1}$ to~$\lab{9}$ for some vertex $m_1$ then it is the block quiver of a Fano polygon. Conversely, if $Q_b$ is the block quiver of a Fano polygon then there is a vertex $m_1$ for which conditions~$\lab{1}$ to~$\lab{9}$ are satisfied.
\end{thm}

\begin{proof}Unless otherwise stated, all indices will be taken modulo $n$. Suppose that $Q_b$ satisfies conditions~$\lab{1}$ to~$\lab{9}$ for some vertex $m_1$. Let $N = \Z^2$ with the standard orientation. Define $v_i:=(x_i,y_i) \in \NQ$ and $m^{i}:=(s_i,t_i)^{t} \in M$ for $i = 1,\ldots,n$. Note that $m^{k}$ is defined if and only if $m_k$ is defined. Otherwise, $m^{k}$ will remain undefined.

The first step is to show that $P:=\conv{v_1,\ldots,v_n} \subset \NQ$ is a Fano polygon. Both the $v_i$ and $m^{i}$ are primitive lattice vectors by~$\lab{5}$,~$\lab{7a}$ and~$\lab{7b}$. A direct calculation shows that:
\begin{equation}\label{eqn:inner_normal_equations_general}
    \langle m^{i},v_i \rangle = \langle m^{i},v_{i+1} \rangle = -\ell_i \text{ for } i=1,\ldots,n.
\end{equation}
Note that the case $i = r$ in~\eqref{eqn:inner_normal_equations_general} is equivalent to condition~$\lab{6}$ and we have used condition~$\lab{4}$ when $i = k$ (if $m_k$ exists). In the present notation, condition~$\lab{9}$ becomes $\langle m^{i},v_j \rangle > -\ell_i$ whenever $j \neq i,i+1$. Thus, $P$ is a two-dimensional lattice polygon (the intersection of finitely many half spaces) with primitive vertices $v_1,\ldots,v_n$. It also follows from conditions~$\lab{9}$ and~$\lab{3}$ that the origin lies in the strict interior of $P$, so that $P$ is a Fano polygon. The remainder of the proof now follows in a similar manner to that of Proposition~\ref{prop:reconstruction_triangles}.
\end{proof}

\section{Remarks on Higher Dimensions}\label{sec:higher_dimensions}

The notion of standard refinement, introduced in Section~\ref{subsec:background}, is fundamental to our definition of polygonal quivers. To define a standard refinement in any given dimension, one must first classify the smallest cones of that dimension for which a combinatorial mutation exists. In dimension $2$ these are precisely the primitive $T$-cones~\cite[Lemma 3.2]{AThesis} and by~\cite[Corollary~2.6]{AK14} they correspond to primitive $T$-singularities, introduced in~\cite{KS88}.

In dimensions greater than $2$, there is at present neither a classification of minimally mutable cones, nor a good understanding of $T$-singularities. This is the main obstruction to defining higher dimensional analogues of polygonal quivers. Nevertheless, it is natural to expect that the higher dimensional theory will possess similar features to those seen in dimension $2$. In particular, the analogue of Proposition~\ref{prop:quiv_mutation_commutes_with_comb_mutation_at_Tcones} should hold: higher polygonal quivers should admit a notion of mutation that is compatible with higher dimensional combinatorial mutations.

A Fano polytope of dimension $n \geq 3$ can admit \emph{nontrivial} combinatorial mutations of codimensions $1,\ldots,n-1$. Here, the codimension of a mutation means $n$ minus the dimension of the factor for that mutation. See~\cite{ACGK12} for the definition of factor. This observation suggests that the `higher polygonal quiver' of an $n$-dimensional Fano polytope $P$ should be an oriented simplicial complex $K(P)$, with decorated faces. A codimension $d$ mutation of $P$ should correspond to a `mutation' of $K(P)$ at face(s) of dimension $d-1$, where $d \in \{1,\ldots,n-1\}$.

In contrast to a standard refinement, the spanning fan of a Fano polytope is well-defined in all dimensions (the definition is the same as that given in Section~\ref{subsec:background}). Thus, following Definition~\ref{def:bquiv_definition}, we may define the block complex $K_b(P)$ directly, bypassing the difficulties discussed above. For brevity, we will restrict ourselves to defining the underlying complex of $K_b(P)$, leaving considerations such as face labels and mutations to future work.

\begin{definition}Let $N \cong \Z^{n}$ be a lattice of rank $n$ with a chosen orientation. For any Fano polytope $P \subset \NQ$, the vertex set of the (undecorated) block complex $K_b(P) = K_b(N,P)$ is the set of inner normal vectors of maximal cones in the spanning fan $\spanf{P}$:
\[
    \verts{K_b(P)} = K^{0}_b(P) := \{m_\sigma \in \Hom{N,\Z} \mid \sigma \text{ is a maximal cone in }\spanf{P}\}.
\]
For every $n$-element subset $\{m_1,\ldots,m_n\} \subset K^{0}_b(P)$, the number of $(n-1)$-simplices in $K_b(P)$ with vertex set $\{m_1,\ldots,m_n\}$ is $|\det(m_1,\ldots,m_n)|$. Every one of these simplices carries the same orientation, chosen by (re-)numbering the vertices so that $\det(m_1,\ldots,m_n)$ is positive.
\end{definition}

\begin{example}\label{eg:block_complex_P3_P1119}Let $N = \Z^{3}$ with the standard orientation. Let $P_1 \subset \Q^3$ be the Fano polytope of $\Proj^3$ with vertex set $\{(-1,-1,-1),(-1,0,-1),(0,-1,-1),(2,2,3)\}$. The inner normal vectors are $m_1 = (0,0,1)^{t}, m_2 = (4,0,-3)^{t}, m_3 = (0,4,-3)^{t}$ and $m_4 = (-4,-4,5)^{t}$. The complex $K_b(P_1)$ is shown in Figure~\ref{fig:P3_P1119_block_complexes} (left), with vertices $m_4$ identified.

\begin{figure}[h!]
    \centering
    \begin{tikzpicture}[thick,scale=0.8, every node/.style={transform shape}]
    \usetikzlibrary{backgrounds}
       \coordinate [label={above:\small $m_1$}] (K1v1) at (-3, 1);
            \draw [fill=black] (K1v1) circle (0.07);
      \coordinate [label={left:\small $m_2$}] (K1v2) at (-4, -0.5);
            \draw [fill=black] (K1v2) circle (0.07);
      \coordinate [label={right:\small $m_3$}] (K1v3) at (-2, -0.5);
            \draw [fill=black] (K1v3) circle (0.07);
      \coordinate [label={right:\small $m_4$}] (K1v4r) at (-1, 1);
            \draw [fill=black] (K1v4r) circle (0.07);
      \coordinate [label={left:\small $m_4$}] (K1v4l) at (-5, 1);
            \draw [fill=black] (K1v4l) circle (0.07);
      \coordinate [label={below:\small $m_4$}] (K1v4b) at (-3,-2);
            \draw [fill=black] (K1v4b) circle (0.07);

       \coordinate [label={\small $16$}] (K1_123) at (-3, -0.25);
       \coordinate [label={\small $16$}] (K1_124l) at (-4.15, 0.25);
       \coordinate [label={\small $16$}] (K1_124r) at (-1.85, 0.25);
       \coordinate [label={\small $16$}] (K1_234b) at (-3, -1.5);

      \draw [middlearrow={latex}] (K1v4l) -- (K1v2);
      \draw [middlearrow={latex}] (K1v2) -- (K1v4b);
      \draw [middlearrow={latex}] (K1v4b) -- (K1v3);
      \draw [middlearrow={latex}] (K1v3) -- (K1v4r);
      \draw [middlearrow={latex}] (K1v4r) -- (K1v1);
      \draw [middlearrow={latex}] (K1v1) -- (K1v4l);
      \draw[bend left=16,middlearrow={latex}]  (K1v1) to node [auto] {} (K1v2);
      \draw[bend left=16,middlearrow={latex}]  (K1v2) to node [auto] {} (K1v1);
      \draw[bend left=16,middlearrow={latex}]  (K1v1) to node [auto] {} (K1v3);
      \draw[bend left=16,middlearrow={latex}]  (K1v3) to node [auto] {} (K1v1);
      \draw[bend left=16,middlearrow={latex}]  (K1v2) to node [auto] {} (K1v3);
      \draw[bend left=16,middlearrow={latex}]  (K1v3) to node [auto] {} (K1v2);
      \begin{scope}[on background layer]
            \path [fill=lightgray!50] (K1v2) to [bend left=16] (K1v1) -- (K1v4l) -- (K1v2);
            \path [fill=lightgray!50] (K1v1) to [bend left=16] (K1v2) to [bend left=16] (K1v3) to [bend left=16] (K1v1);
            \path [fill=lightgray!50] (K1v1) to [bend left=16] (K1v3) -- (K1v4r) -- (K1v1);
            \path [fill=lightgray!50] (K1v2) -- (K1v4b) -- (K1v3) to [bend left=16] (K1v2);
      \end{scope}

      \coordinate [label={above:\small $m_1$}] (K2v1) at (3, 1);
            \draw [fill=black] (K2v1) circle (0.07);
      \coordinate [label={left:\small $m_2$}] (K2v2) at (2, -0.5);
            \draw [fill=black] (K2v2) circle (0.07);
      \coordinate [label={right:\small $m_3$}] (K2v3) at (4, -0.5);
            \draw [fill=black] (K2v3) circle (0.07);
      \coordinate [label={right:\small $m_4$}] (K2v4r) at (5, 1);
            \draw [fill=black] (K2v4r) circle (0.07);
      \coordinate [label={left:\small $m_4$}] (K2v4l) at (1, 1);
            \draw [fill=black] (K2v4l) circle (0.07);
      \coordinate [label={below:\small $m_4$}] (K2v4b) at (3,-2);
            \draw [fill=black] (K2v4b) circle (0.07);

       \coordinate [label={\small $16$}] (K2_123) at (3, -0.25);
       \coordinate [label={\small $16$}] (K2_124r) at (4.15, 0.25);
       \coordinate [label={\small $16$}] (K2_124l) at (1.85, 0.25);
       \coordinate [label={\small $48$}] (K2_234b) at (3, -1.5);

      \draw [middlearrow={latex}] (K2v4l) -- (K2v1);
      \draw [middlearrow={latex}] (K2v1) -- (K2v4r);
      \draw [middlearrow={latex}] (K2v4r) -- (K2v3);
      \draw [middlearrow={latex}] (K2v3) -- (K2v4b);
      \draw [middlearrow={latex}] (K2v4b) -- (K2v2);
      \draw [middlearrow={latex}] (K2v2) -- (K2v4l);
      \draw[bend right=16,middlearrow={latex}]  (K2v1) to node [auto] {} (K2v2);
      \draw[bend right=16,middlearrow={latex}]  (K2v2) to node [auto] {} (K2v1);
      \draw[bend right=16,middlearrow={latex}]  (K2v1) to node [auto] {} (K2v3);
      \draw[bend right=16,middlearrow={latex}]  (K2v3) to node [auto] {} (K2v1);
      \draw[bend right=16,middlearrow={latex}]  (K2v2) to node [auto] {} (K2v3);
      \draw[bend right=16,middlearrow={latex}]  (K2v3) to node [auto] {} (K2v2);
      \begin{scope}[on background layer]
            \path [fill=lightgray!50] (K2v2) to [bend left=16] (K2v1) -- (K2v4l) -- (K2v2);
            \path [fill=lightgray!50] (K2v1) to [bend left=16] (K2v2) to [bend left=16] (K2v3) to [bend left=16] (K2v1);
            \path [fill=lightgray!50] (K2v1) to [bend left=16] (K2v3) -- (K2v4r) -- (K2v1);
            \path [fill=lightgray!50] (K2v2) -- (K2v4b) -- (K2v3) to [bend left=16] (K2v2);
      \end{scope}
\end{tikzpicture}
\caption{The Block Complexes for $\Proj^{3}$ (left) and $\Proj(1,1,1,9)/\mu_3$ (right).}\label{fig:P3_P1119_block_complexes}
\end{figure}
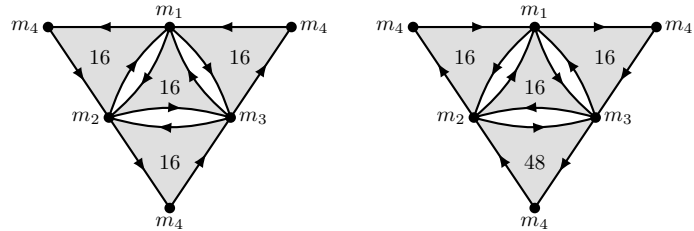
Note that $\Proj^{3}$ is smooth, so every maximal cone of the spanning fan $\Sigma_{P_1}$ is (trivially) a primitive $T$-cone. In particular, $\Sigma_{P_1}$ is its own standard refinement and $K_b(P_1) = K(P_1)$ in this example. Next, consider $P_2 \subset \Q^{3}$ with vertex set $\{(-1,-1,-1),(2,5,3),(5,2,3),(2,2,3)\}$. The spanning fan of $P_2$ defines a quotient of $\Proj(1,1,1,9)$ by the cyclic group $\mu_3$ and $P_2$ is a codimension-1 mutation of $P_1$ with respect to the width vector $(0,0,1)^{t}$ and factor $\conv{\orig, (0,1,0), (1,0,0)} \subset \Q^{3}$. The inner normals are $m_1 = (0,0,-1)^{t}, m_2 = (4,0,-3)^{t}, m_3 = (0,4,-3)^{t}, m_4 = (-4,-4,9)^{t}$ and $K_b(P_2)$ is shown in Figure~\ref{fig:P3_P1119_block_complexes} (right), with $m_4$ identified.
\end{example}

\subsection*{Acknowledgements}
We thank the following colleagues for many interesting conversations: P.~Bousseau, T.~Coates, M.~Kontsevich, V.~Pestun, K.~Rietsch and T.~Sutherland. This work was carried out during the author's stay as an EPSRC-funded William Hodge Fellow at the Institut des Hautes \'{E}tudes Scientifiques. Much of the writing was completed while the author was a Visiting Research Fellow at King's College London. We thank both the IH\'{E}S and King's College London for excellent working conditions.

\bibliographystyle{abbrv}
\bibliography{bibliography}
\end{document}